\documentclass[letterpaper]{amsart}

\usepackage[alphabetic,initials,nobysame]{amsrefs}
\usepackage{amssymb,mathtools,mathrsfs,enumitem,color,comment}
\usepackage{pgfplots} \pgfplotsset{compat=1.18}
\usetikzlibrary{calc,shapes.geometric}

\usepackage[bookmarksdepth=2]{hyperref}

\BibSpec{collection.article}{%
	+{}  {\PrintAuthors}				{author}
	+{,} { \textit}                  	{title}
	+{.} { }                         	{part}
	+{:} { \textit}                  	{subtitle}
	+{,} { \PrintContributions}      	{contribution}
	+{,} { \PrintConference}         	{conference}
	+{}  {\PrintBook}                	{book}
	+{,} { }                         	{booktitle}
	+{,} { }						 	{series}
	+{,} { }						 	{publisher}
	+{,} { }						 	{address}
	+{,} { \PrintDateB}              	{date}
	+{,} { pp.~}                     	{pages}
	+{,} { }                         	{status}
	+{,} { \PrintDOI}                	{doi}
	+{,} { available at \eprint}     	{eprint}
	+{}  { \parenthesize}            	{language}
	+{}  { \PrintTranslation}        	{translation}
	+{;} { \PrintReprint}            	{reprint}
	+{.} { }                         	{note}
	+{.} {}                          	{transition}
	+{}  {\SentenceSpace \PrintReviews} {review}
}

\theoremstyle{plain}
\newtheorem{theorem}{Theorem}
\newtheorem{lemma}[theorem]{Lemma}
\newtheorem{proposition}[theorem]{Proposition}
\newtheorem{corollary}[theorem]{Corollary}

\theoremstyle{definition}

\theoremstyle{remark}
\newtheorem{remark}[theorem]{Remark}

\newtheorem{problem}[theorem]{Problem}

\numberwithin{equation}{section}
\numberwithin{theorem}{section}

\renewcommand{\bar}{\overline}
\newcommand{\R}{\mathbb R}

\newcommand{\C}{\mathbb C}
\newcommand{\D}{\mathbb D}

\newcommand{\N}{\mathbb N}

\DeclareMathOperator{\dist}{{\mathrm{dist}}}
\DeclareMathOperator{\diam}{{\mathrm{diam}}}

\DeclareMathOperator{\inter}{{\mathrm{int}}}
\DeclareMathOperator{\re}{{\mathrm{Re}}}
\DeclareMathOperator{\im}{{\mathrm{Im}}}

\DeclareMathOperator{\Mod}{Mod}

\DeclareMathOperator{\loc}{\mathrm{loc}}

\DeclareMathOperator{\llc}{\mathit{LLC}}

\title{Quasiconformal characterization of Schottky sets}
\author{Dimitrios Ntalampekos}
\address{Department of Mathematics, Aristotle University of Thessaloniki, Thessaloniki, 54152, Greece.}
\thanks{The author is partially supported by the ERC Starting Grant, Grant Agreement no. 101214615, GRComPaS}
\email{dntalam@math.auth.gr}
\date{\today}

\begin{document}

	\begin{abstract}
	The complement of the union of a collection of disjoint open disks in the $2$-sphere is called a Schottky set. We prove that a subset $S$ of the $2$-sphere is quasiconformally equivalent to a Schottky set if and only if every pair of complementary components of $S$ can be mapped to a pair of open disks with a uniformly quasiconformal homeomorphism of the sphere. Our theorem applies to Sierpi\'nski carpets and gaskets, yielding for the first time a general quasiconformal uniformization result for gaskets. Moreover, it contains Bonk's uniformization result for carpets as a special case and does not rely on the condition of uniform relative separation that is used in relevant works. 
	\end{abstract}

\maketitle

\setcounter{tocdepth}{1}
\tableofcontents

\section{Introduction}

In this work we provide a necessary and sufficient condition so that a compact set in the Riemann sphere is quasiconformally equivalent to a Schottky set. We discuss some background before stating the main result.

\subsection{Background}
A \textit{(geometric) disk} in the Riemann sphere $\widehat \C$ is an open ball $B(a,r)$ in the spherical metric, where $0<r<\pi$.  The complement of the union of a collection of disjoint open (geometric) disks in the Riemann sphere is called a \textit{Schottky set}. Special cases of Schottky sets include closures of circle domains and round Sierpi\'nski carpets. A \textit{circle domain} is a domain in the Riemann sphere whose complementary components are closed disks or points. A \textit{Sierpi\'nski carpet} is a compact set in the sphere with empty interior whose complementary components are countably many Jordan regions, called \textit{peripheral disks}, with disjoint closures and with diameters shrinking to zero. A fundamental result of Whyburn \cite{Whyburn:theorem} states that all Sierpi\'nski carpets are homeomorphic to each other. A \textit{round} Sierpi\'nski carpet has the additional requirement that its peripheral disks are geometric disks.

It is desirable to find a quasiconformal homeomorphism of the sphere that transforms a compact subset of the sphere to a Schottky set or to a round Sierpi\'nski carpet, due to the strong geometric properties of the latter classes of sets. For example, they enjoy strong rigidity properties and they can be studied using dynamical methods with the aid of groups generated by reflections in the complementary disks; for instance, see \cites{HeSchramm:Rigidity, BonkKleinerMerenkov:schottky, Merenkov:relativeSchottky, Merenkov:localrigiditySchottky, BonkLyubichMerenkov:carpetJulia, NtalampekosYounsi:rigidity, Ntalampekos:rigidity_cned}.

A first uniformization result for round carpets was proved by McMullen \cite{McMullen:teichmuller}, who showed that if the limit set of a convex cocompact Kleinian group is a Sierpi\'nski carpet, then it is quasiconformally equivalent to a round carpet. Herron and Koskela \cite{HerronKoskela:QEDcircledomains} proved that each uniform domain can be transformed to a circle domain via a quasiconformal homeomorphism of the sphere. In particular, the closure of a uniform domain can be quasiconformally uniformized by a Schottky set. This generalized an earlier result of Herron \cite{Herron:uniform}. 

A few years later Schramm \cite{Schramm:transboundary} devised a powerful tool for the study of uniformization and rigidity problems in the plane, known as transboundary modulus; see Section \ref{section:transboundary}. This tool turned out to be of paramount importance for the study of Koebe's conjecture, a deep open problem in complex analysis that is closely related to the uniformization of Sierpi\'nski carpets. \textit{Koebe's conjecture} asserts that for every domain $\Omega\subset \widehat \C$ there exists a conformal map from $\Omega$ onto a circle domain. In contrast, in this work we are interested in quasiconformal homeomorphisms \textit{of the entire sphere} that transform a given set to a Schottky set or a carpet. 

Bonk \cite{Bonk:uniformization} used the full potential of transboundary modulus in order to provide a sufficient criterion for a Sierpi\'nski carpet to be quasiconformally equivalent to a round carpet. Let $S\subset \widehat \C$ be a Sierpi\'nski carpet and let $\{U_i\}_{i\in \N}$ be its peripheral disks. We say that the peripheral disks of $S$ are \textit{uniform quasidisks} if, for each $i\in I$, $U_i$ can be mapped to a geometric disk by a $K$-quasiconformal homeomorphism of $\widehat \C$, where $K\geq 1$ is a uniform constant. We say that the peripheral disks of $S$ are \textit{uniformly relatively separated} if there exists a constant $\delta>0$ such that the \textit{relative distance} of $U_i$ and $U_j$, defined by
\begin{align*}
\Delta(U_i,U_j) = \frac{\dist(U_i,U_j)}{\min\{\diam(U_i),\diam(U_j)\}},
\end{align*}
is bounded from below by $\delta$ for each $i\neq j$. Bonk proved that if the peripheral disks of $S$ are uniformly relatively separated uniform quasidisks, then there exists a quasiconformal homeomorphism of $\widehat \C$ that maps $S$ onto a round carpet. Trivially, the assumption of uniform quasidisks is necessary for the conclusion. The uniform relative separation prevents large peripheral disks from being too close to each other. This condition is obviously not necessary, since round carpets need not satisfy it. 

Bonk's excellent criterion was exploited in rigidity problems that appear in the study of the standard Sierpi\'nski carpet and of Julia sets \cites{BonkMerenkov:rigidity, BonkLyubichMerenkov:carpetJulia, BonkMerenkov:rigiditySpCarpets}, and was extended to non-planar carpets \cites{MerenkovWildrick:uniformization, Rehmert:thesis}. However, the removal or replacement of the assumption of uniform relative separation with a weaker assumption remained elusive.  It was observed by the author in \cite{Ntalampekos:CarpetsThesis}*{Prop.\ 3.1.6} that the assumption of uniform relative separation is necessary for the quasiconformal transformation of a round carpet to a square carpet (defined in the obvious manner). This provided some evidence that the result of Bonk is optimal.

Subsequent works of the author \cites{Ntalampekos:CarpetsThesis, Ntalampekos:uniformization_packing} removed the assumption of uniform relative separation (and also weakened the assumption of uniform quasidisks) at the cost of weakening the regularity of the uniformizing map, which transforms an arbitrary carpet to a round or square carpet. For instance, the uniformizing maps in \cite{Ntalampekos:CarpetsThesis} are termed \textit{carpet-quasiconformal} and are required to quasipreserve a generalization of transboundary modulus, called \textit{carpet modulus}. In particular, the uniformizing maps that appear in \cites{Ntalampekos:CarpetsThesis, Ntalampekos:uniformization_packing} are not quasiconformal in the usual sense in general. These results, therefore, provided further evidence that the separation assumption of Bonk cannot be relaxed. 

The quasiconformal uniformization of carpets whose peripheral disks are too close to each other, as well as, of limiting objects in which ``peripheral disks" can touch each other had withstood extensive investigation. In particular, a general uniformization result for sets that resemble the Sierpi\'nski gasket (in the sense that complementary components can touch each other) seemed to be out of reach. 

However, very recently, the author and Luo \cite{LuoNtalampekos:gasket} characterized, under some mild conditions, Julia sets of rational maps that can be quasiconformally uniformized by round \textit{gaskets}; see \cite{LuoNtalampekos:gasket}*{Definition 1.1} for the definition of a topological gasket. Specifically, it is shown that a Julia set that does not contain critical points can be quasiconformally uniformized by a round gasket if and only if it is a \textit{fat gasket}, i.e., boundaries of Fatou components intersect tangentially. Yet, the result and its proof rely heavily on the structure of Julia sets and on complex dynamical methods, and it is clear that the proof does not extend to arbitrary gaskets not arising from a dynamical framework.

In this work we overcome this obstacle and, as our main theorem, we provide a characterization of compact subsets of the sphere that are quasiconformally equivalent to Schottky sets, removing entirely the separation assumption of Bonk.

\subsection{Main result}
Let $U,V\subset \widehat \C$ be disjoint Jordan regions with distinct boundaries, i.e., $\partial U\neq \partial V$. We say that the pair $U,V$ is \textit{quasiconformally circularizable} if there exists a quasiconformal homeomorphism $\phi\colon \widehat \C\to \widehat \C$ such that $\phi(U)$ and $\phi(V)$ are (geometric) disks. If $\phi$ can be taken to be $K$-quasiconformal for some $K\geq 1$, then we say that the pair $U,V$ is $K$-quasiconformally circularizable. Note that the closures of $U$ and $V$ are allowed to intersect each other in at most one point. 

Let $\{U_i\}_{i\in I}$ be a collection of at least two disjoint Jordan regions in $\widehat \C$ with distinct boundaries (so we exclude the trivial scenario that the collection consists of only two Jordan regions that share the same boundary). We say that the regions $U_i$, $i\in I$, are \textit{uniformly quasiconformally pairwise circularizable} if there exists $K\geq 1$ such that every pair $U_i,U_j$, $i\neq j$, is $K$-quasiconformally circularizable. We now state our main theorem.

\begin{theorem}\label{theorem:main}
Let $\{U_i\}_{i\in I}$ be a collection of at least two disjoint Jordan regions in $\widehat \C$ with distinct boundaries. The following are quantitatively equivalent.
\begin{enumerate}[label=\normalfont(\roman*)]
\item The regions $U_i$, $i\in I$, are uniformly quasiconformally pairwise circularizable.

\item There exists a collection of disjoint disks $\{D_i\}_{i\in I}$ in $\widehat \C$ and a quasiconformal homeomorphism $f\colon \widehat \C\to \widehat \C$ that maps the compact set $S=\widehat \C\setminus \bigcup_{i\in I} U_i$ onto the Schottky set $T=\widehat \C\setminus \bigcup_{i\in I}D_i$ and is $1$-quasiconformal on $S$. 
\end{enumerate}
In this case, the map $f|_S$ is unique up to postcomposition with a conformal automorphism of $\widehat \C$. 
\end{theorem}

Here, a quasiconformal map $f\colon U\to V$ between open subsets of $\widehat \C$ is said to be $1$-quasiconformal on a measurable set $S\subset U$ if  $\|Df(z)\|^2=J_f(z)$ for a.e.\ $z\in S$. The statement is quantitative in the usual sense. Namely, the quasiconformal distortion bound of the uniformizing map and the constant of the quasiconformal circularization assumption depend only on each other.

We note that the assumptions of Bonk on uniform quasidisks that are uniformly relatively separated are stronger than the quasiconformal circularization assumption; see Proposition \ref{proposition:urs}. Thus, our result implies Bonk's theorem \cite{Bonk:uniformization}*{Theorem 1.1}. Also, the uniqueness in the last statement of our theorem is based on and extends results from \cite{BonkKleinerMerenkov:schottky}, regarding the rigidity of Schottky sets of area zero.

One of the novelties of the proof of existence in Theorem \ref{theorem:main} is that it relies on properties of groups generated by reflections in disjoint circles, known as Schottky groups; see Section \ref{section:schottky}. Such techniques have been used extensively in the past for the proof of significant rigidity results (e.g., see \cites{HeSchramm:Rigidity, BonkKleinerMerenkov:schottky}). It is quite remarkable that a reflection argument can provide additional leverage in existence results and we expect that the methods introduced here can result in further progress in other uniformization problems in the complex plane. Furthermore, our theorem is expected to have considerable consequences in complex dynamics, yielding uniformization and rigidity results beyond the hyperbolic setting and without any limitation on the separation of Fatou components.

We remark that a priori it is not even clear whether there exists a homeomorphism that maps the set $S$ of Theorem \ref{theorem:main} onto a Schottky set. A necessary condition is that for any $i\neq j$ the closures $\bar{U_i}$ and $\bar{U_j}$ are either disjoint or intersect at a single point; in our case, this is guaranteed by the assumption that $U_i\cup U_j$ can be mapped to the union of two disks with a homeomorphism of the sphere. A more subtle condition is that no three regions $\bar{U_i}$, $i\in I$, can meet at a point; this is also guaranteed by the quasiconformal circularization assumption, as shown in Lemma \ref{lemma:three_regions}. Yet another requirement is that for each $\varepsilon>0$ there exist at most finitely many $i\in I$ such that $\diam (U_i)>\varepsilon$; this follows from the assumption that the regions $U_i$, $i\in I$, are disjoint uniform quasidisks. Are these necessary conditions also sufficient?

\begin{problem}
Find a topological characterization of Schottky sets.
\end{problem}

We also state a consequence of Theorem \ref{theorem:main} regarding circle domains. 
\begin{corollary}
A domain $\Omega\subset \widehat \C$ can be mapped onto a circle domain by a conformal map that is the restriction of a quasiconformal homeomorphism of $\widehat\C$ if and only if the components of $\widehat \C\setminus \bar \Omega$ are Jordan regions that are uniformly quasiconformally pairwise circularizable.
\end{corollary}

This corollary resolves a problem studied by Herron and Koskela \cites{Herron:uniform, HerronKoskela:QEDcircledomains}, who provided sufficient conditions for the transformation of a domain onto a circle domain with a quasiconformal homeomorphism of the sphere.

\subsection{Proof outline}
We outline the proof of Theorem \ref{theorem:main}. Note that (ii) trivially implies (i), so we sketch the proof of the reverse implication.

\subsubsection*{Preliminary reductions}
The idea of the proof is to consider a finite subset $J$ of $I$ and find a quasiconformal homeomorphism of the sphere that maps $\widehat \C\setminus \bigcup_{i\in J} U_i$ onto a Schottky set. Then one wishes to show equicontinuity of such maps over all finite sets $J\subset I$ and pass to a subsequential limit in order to obtain the desired uniformizing map.

Even at the very first step we face an obstacle that is not present in other works. Koebe's uniformization theorem allows one to map a finitely connected domain conformally onto a circle domain. If the regions $U_i$, $i\in J$, had disjoint closures, then $G=\widehat \C\setminus \bigcup_{i\in J}\bar{U_i}$ would be a finitely connected domain. For instance, this is the case in the setting of Sierpi\'nski carpets. However, in our setting the closures $\bar{U_i}$, $i\in J$, can intersect. We overcome this by finding an exhaustion $U_i(n)$, $n\in \N$, of each $U_i$, $i\in J$, so that for each $n\in \N$ the regions ${U_i(n)}$, $i\in J$, have disjoint closures and most importantly they are still uniformly quasiconformally pairwise circularizable; see Proposition \ref{prop:disjoint_closures}. This step is technical and the obvious approaches do not seem to work. On the other hand, note that Theorem \ref{theorem:main}, once proved, implies the existence of such exhaustions immediately. 

In the next step, assuming that the closures $\bar{U_i}$, $i\in J$, are disjoint, we apply Koebe's uniformization theorem to find a conformal map from the domain $G=\widehat \C\setminus \bigcup_{i\in J}\bar{U_i}$ onto a circle domain. This conformal map extends to the closure $\bar G$ and we assume that $\partial U_i$ is mapped to a circle that bounds a closed disk $K_i\subset \widehat \C\setminus f(G)$ for each $i\in J$. We wish to find an extension of $f$ to a uniformly quasiconformal homeomorphism of $\widehat \C$. Then standard equicontinuity properties of uniformly quasiconformal maps will allow us to pass to the limit and prove the uniformization result.

By a consequence of the Beurling--Ahlfors extension, it suffices to show that, after suitable normalizations, the conformal map $f$ is uniformly quasisymmetric in the finitely connected domain $G$.

\subsubsection*{Proving quasisymmetric behavior via conformal modulus}
Classical results (see e.g.\ \cite{Heinonen:metric}*{Theorem 11.14}) imply that (quasi)conformal maps are quasisymmetric away from the boundary. This is not quite satisfactory in our case, since we need $f$ to be quasisymmetric on all of $G$. However, the method of proof is more valuable. 

According to that method, one needs to show that if two continua $E,F\subset G$ are relatively close to each other (meaning that the relative distance $\Delta(E,F)$ is bounded from above), then the family of curves in $G$ connecting them, denoted by $\Gamma(E,F;G)$, has large conformal modulus. On the other hand, if the images $E'=f(E)$, $F'=f(F)$ are not relatively close to each other (that is, $\Delta(E',F')$ is large), then one should prove that the conformal modulus of the family $\Gamma(E',F';f(G))$ is small. Assuming these bounds, the conformal invariance of modulus implies that if $\Delta(E,F)$ is small, then $\Delta(E',F')$ must also be small in a quantitative way. This is sufficient to show that $f$ is quasisymmetric by a classical result of V\"ais\"al\"a. 

The desired upper bound for the conformal modulus of $\Gamma(E',F';f(G))$ is immediate from the fact that the modulus of curves connecting the boundary components of a Euclidean annulus with inner radius $r$ and outer radius $R$ is $2\pi (\log(R/r))^{-1}$; this converges to $0$ as $r\to 0$. However, the modulus of curves in $G$ that connect $E$ and $F$ can be very small even when $\Delta(E,F)$ is small. This is because passages between complementary components of $G$ can be very narrow; see Figure \ref{figure:modulus_small}.

\begin{figure}
\begin{tikzpicture}

\begin{scope}
\node[regular polygon, rounded corners=5pt, draw, regular polygon sides=7, fill=black!15, minimum size=2cm] (p) at (0,0) {};
\node[regular polygon, rounded corners=3pt, draw, regular polygon sides=9, fill=black!15, minimum size=2cm] (p) at (2.1,0) {};

\draw[color=red, line width=1.5pt] (0.5,1)-- node[above] {$E$}(1.5,1);
\draw[color=red, line width=1.5pt] (0.5,-1)-- node[below] {$F$}(1.5,-1);

\draw[rounded corners, red] (0.7,1)--(1,0.5)--(0.95,0.2)--(1.05,-0.2)--(0.7,-1);
\draw[rounded corners, red,xshift=0.1cm] (0.7,1)--(1,0.5)--(0.95,0.2)--(1.05,-0.2)--(0.7,-1);
\draw[rounded corners, red, rotate=180, xshift=-2cm] (0.7,1)--(1,0.5)--(0.95,0.2)--(1.05,-0.2)--(0.7,-1);
\draw[rounded corners, red, rotate=180, xshift=-2.1cm] (0.7,1)--(1,0.5)--(0.95,0.2)--(1.05,-0.2)--(0.7,-1);
\draw[rounded corners, red] (1.1,1)to[out=260, in=80](1.02,-1);
\end{scope}

\begin{scope}[shift={(5,0)}]
\node[regular polygon, rounded corners=5pt, draw, regular polygon sides=7, fill=black!15, minimum size=2cm] (p) at (0,0) {};
\node[regular polygon, rounded corners=3pt, draw, regular polygon sides=9, fill=black!15, minimum size=2cm] (p) at (2.1,0) {};

\draw[color=red, line width=1.5pt] (0.5,1)-- node[above] {$E$}(1.5,1);
\draw[color=red, line width=1.5pt] (0.5,-1)-- node[below] {$F$}(1.5,-1);

\draw[rounded corners, red] (0.7,1)to[out=260, in=80](0.7,-1);
\draw[rounded corners, red, xshift=0.25cm] (0.7,1)to[out=260, in=80](0.7,-1);
\draw[rounded corners, red,xshift=0.1cm] (0.7,1)to[out=260, in=80](0.7,-1);
\draw[rounded corners, red, rotate=180, xshift=-2cm] (0.7,1)to[out=260, in=80](0.7,-1);
\draw[rounded corners, red, rotate=180, xshift=-2.1cm] (0.7,1)to[out=260, in=80](0.7,-1);
\draw[rounded corners, red] (1.1,1)to[out=260, in=80](1.02,-1);
\end{scope}
\end{tikzpicture}
\caption{The conformal modulus of curves connecting $E$ and $F$ is small due to a narrow passage, but the transboundary modulus is large.}\label{figure:modulus_small}
\end{figure}
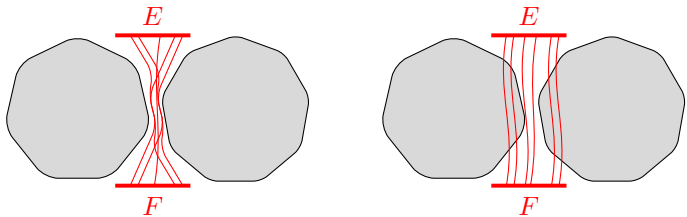

\subsubsection*{Proving quasisymmetric behavior via transboundary modulus}
In order to bound modulus away from zero, we need to switch to a different notion of modulus, the transboundary modulus that was introduced by Schramm; see Section \ref{section:transboundary} for the definition. The striking advantage of transboundary modulus is that it allows for curves to exit the domain $G$ and travel through complementary components, while maintaining conformal invariance at the same time. If the geometry of the complementary components of $G$ is nice, which is the case in our setting since each $U_i$, $i\in J$, is a uniform quasidisk, then the transboundary modulus of the family of curves in $\widehat \C$ (rather than in $G$) connecting $E$ and $F$ is large, whenever $\Delta(E,F)$ is bounded from above; see Figure \ref{figure:modulus_small}. In modern terminology, transboundary modulus has the Loewner property; see Section \ref{section:lower}. We prove this in Proposition \ref{prop:modulus_lower}. We note that Bonk obtains a similar result \cite{Bonk:uniformization}*{Prop.\ 8.1}, but he imposes the uniform relative separation assumption, which is not available here.

So, based on the method we discussed, now one needs to go back to the images $E',F'$ and ensure that whenever $\Delta(E',F')$ is large, then the transboundary modulus of the family of curves $\Gamma(E',F';\widehat \C)$ is small. Unfortunately, this is not the case in general! It was observed by Bonk that the transboundary modulus can be actually large. This is because there can be some large disks that function as bridges between two relatively small continua $E'$ and $F'$; see Figure \ref{figure:modulus_large}. Bonk observes that there can be a uniformly bounded number of disks (in fact, at most $2$) that can serve as bridges between $E'$ and $F'$ and cause transboundary modulus to be large. He resolves this issue by restricting to the subfamily of curves in $\widehat \C$ connecting $E'$ and $F'$ and avoiding a bounded number of large disks that function as bridges. 

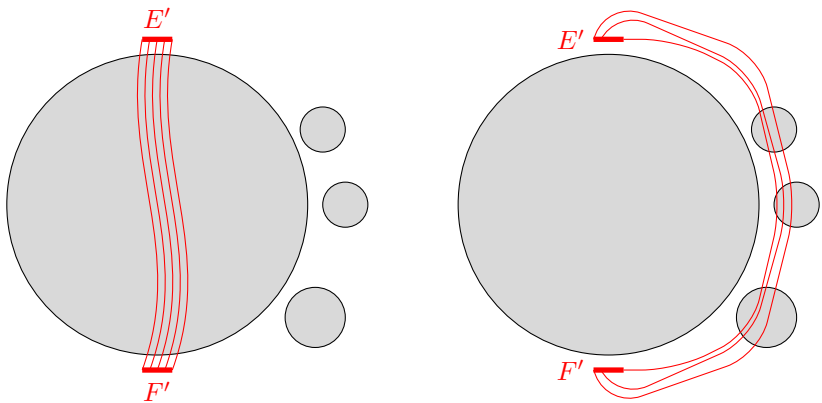
\begin{figure}

\begin{tikzpicture}[scale=1]
\begin{scope}
\draw[fill=black!15] (0,0) circle (2cm);
\draw[color=red, line width=2pt]  (-0.2,2.2)-- node[above]{$E'$}(0.2,2.2);
\draw[color=red, line width=2pt]  (-0.2,-2.2)-- node[below]{$F'$}(0.2,-2.2);

\draw[fill=black!15] (2.5,0) circle (0.3cm);
\draw[fill=black!15] (2.2,1) circle (0.3cm);
\draw[fill=black!15] (2.1,-1.5) circle (0.4cm);

\draw[color=red, rounded corners] (-0.2,2.2)to[out=260,in=70](-0.2,-2.2);
\draw[color=red, rounded corners, xshift=0.1cm] (-0.2,2.2)to[out=260,in=70](-0.2,-2.2);
\draw[color=red, rounded corners,xshift=0.2cm ] (-0.2,2.2)to[out=260,in=70](-0.2,-2.2);
\draw[color=red, rounded corners, xshift=0.3cm] (-0.2,2.2)to[out=260,in=70](-0.2,-2.2);
\draw[color=red, rounded corners, xshift=0.4cm] (-0.2,2.2)to[out=260,in=70](-0.2,-2.2);
\end{scope}

\begin{scope}[shift={(6,0)}]
\draw[fill=black!15] (0,0) circle (2cm);
\draw[color=red, line width=2pt]  (-0.2,2.2)-- node[xshift=-0.5cm]{$E'$}(0.2,2.2);
\draw[color=red, line width=2pt]  (-0.2,-2.2)-- node[xshift=-0.5cm]{$F'$}(0.2,-2.2);

\draw[fill=black!15] (2.5,0) circle (0.3cm);
\draw[fill=black!15] (2.2,1) circle (0.3cm);
\draw[fill=black!15] (2.1,-1.5) circle (0.4cm);

\draw[color=red, rounded corners=13pt] (-0.2,2.2)-- (0,2.7)--(2,2)--(2.5,0)--(2,-2)--(0,-2.7)--(-0.2,-2.2);
\draw[color=red, rounded corners=13pt,xshift=0.1cm] (-0.2,2.2)-- (0,2.6)--(1.8,1.8)--(2.3,0)--(1.8,-1.8)--(0,-2.6)--(-0.2,-2.2);
\draw[color=red, rounded corners=13pt,xshift=0.2cm]  (0,2.2)-- (0.7,2.2)--(1.7,1.7)--(2.1,0)--(1.7,-1.7)--(0.7,-2.2)--(0,-2.2);
\end{scope}
\end{tikzpicture}

\caption{The transboundary modulus of curves connecting $E'$ and $F'$ is large because of a large disk $K_i$ that functions as a bridge. However, if one avoids this one large disk, then the modulus becomes small.}\label{figure:modulus_large}
\end{figure}

It seems that we are trapped in a vicious cycle, because now the question is whether by avoiding a bounded number of complementary components of $G$ we have destroyed our favorable modulus estimates for the family of curves connecting $E$ and $F$; recall the phenomenon in Figure \ref{figure:modulus_small}. Nevertheless, ingeniously Bonk here uses again the relative separation assumption to show that this does not happen \cite{Bonk:uniformization}*{Prop.\ 7.5}. Not surprisingly, it is impossible to obtain a general inequality in the spirit of \cite{Bonk:uniformization}*{Prop.\ 7.5} without the relative separation assumption.

\subsubsection*{Proving quasisymmetric behavior without the separation assumption}
We remark that the method of Bonk shows that any conformal map is quasisymmetric as long as it maps a finitely connected domain whose complementary components are uniformly relatively separated uniform (closed) quasidisks onto a domain whose complementary components are uniform (closed) quasidisks. Thus, no particular property of geometric disks is used, other than the trivial fact that they are quasidisks. Instead, here we use the full strength of the geometry of disks, which is manifested by the existence of anti-conformal reflections along their boundaries. 
 
We show that if $\Delta(E',F')$ is relatively large, there are \textit{at most two} disks $K_{i_1},K_{i_2}$ that are large and are close to $E',F'$ that could potentially serve as bridges in the argument discussed above. To put it differently, if two disks $K_{i_1},K_{i_2}$ are so large that they intersect both boundary components of an annulus, then \textit{all other disks} that are disjoint from $K_{i_1},K_{i_2}$ must be small in a sense with respect to that annulus; see Proposition \ref{prop:annulus_euclidean_general} for a precise statement. Note that the number $2$ is important here. If one tried to prove an analogous statement for squares instead of disks, perhaps the appropriate maximum number of squares that can serve as bridges would be $4$.

Let $U_{i_1}$ and $U_{i_2}$ be the Jordan regions that correspond to $K_{i_1}$ and $K_{i_2}$ under $f^{-1}$, respectively. We now use the quasiconformal circularization assumption to map the pair $U_{i_1},U_{i_2}$ onto geometric disks $V_{i_1},V_{i_2}$ with a quasiconformal homeomorphism $h$ of the sphere. The remaining regions $U_i$, $i\neq i_1,i_2$, are transformed to uniform quasidisks $V_i=h(U_i)$. Let $g=f\circ h^{-1}$. Remember, our task is to show favorable distortion bounds for $g$, and thus for $f$. Namely, if $\Delta(h(E),h(F))$ is small, we wish to show that $\Delta(E',F')$ is also small.  

An important ingredient of the proof is that the reflections of \textit{all other disks} $K_i$, $i\neq i_1,i_2$, inside $K_{i_1}$ and $K_{i_2}$ remain \textit{favorable} and do not create new bridges between $E'$ and $F'$; see Figure \ref{figure:extension}. Informally, here we say that a disk is {favorable} if it is very small relative to the sets $E'$ and $F'$. In addition, all iterated reflections of $K_i$, $i\neq i_1,i_2$, inside $K_{i_1}$ and $K_{i_2}$ remain favorable; equivalently, one can say that the orbit of $K_i$ with respect to the Schottky group generated by reflections in the disks $K_{i_1}$ and $K_{i_2}$ consists of favorable disks that do not serve as bridges. This is an elaborate result proved in  Proposition \ref{prop:reflect}.

On the other side, the orbit of the quasidisks $V_i$, $i\neq i_1,i_2$, with respect to the Schottky group generated by reflections in the disks $V_{i_1}$ and $V_{i_2}$ consists of uniform quasidisks. Moreover, by indefinitely repeated reflections, we may extend $g$ to a quasiconformal map in a larger domain than $h(G)$; see Figure \ref{figure:extension}. We now estimate the transboundary modulus of $\Gamma(h(E),h(F);\widehat \C)$ with respect to this larger domain. By the Loewner property of transboundary modulus (Proposition \ref{prop:modulus_lower}), this modulus is large, provided that $\Delta(h(E),h(F))$ is small.

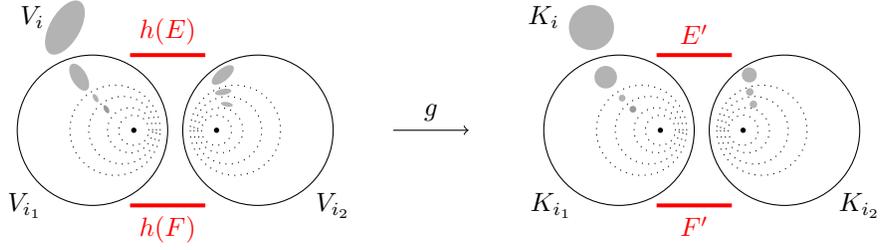
\begin{figure}
\begin{tikzpicture}
\begin{scope}
\draw(0,0) circle (1cm); \node at (-0.9,-1) {$V_{i_1}$};
\draw(2.2,0) circle (1cm); \node at (3.2,-1) {$V_{i_2}$};
\draw[color=red, line width=1.5pt] (0.5,1)-- node[above] {$h(E)$}(1.5,1);
\draw[color=red, line width=1.5pt] (0.5,-1)-- node[below] {$h(F)$}(1.5,-1);

\draw[dotted] (0.3,0) circle (0.6cm);
\draw[dotted] (0.4,0) circle (0.45cm);
\draw[dotted] (0.5,0) circle (0.3cm);
\draw[dotted] (0.55,0) circle (0.2cm);
\fill (0.55,0) circle (1pt);

\draw[dotted, rotate=180, xshift=-2.2cm] (0.3,0) circle (0.6cm);
\draw[dotted, rotate=180, xshift=-2.2cm] (0.4,0) circle (0.45cm);
\draw[dotted, rotate=180, xshift=-2.2cm] (0.5,0) circle (0.3cm);
\draw[dotted, rotate=180, xshift=-2.2cm] (0.55,0) circle (0.2cm);
\fill[rotate=180, xshift=-2.2cm] (0.55,0) circle (1pt);

\node at (-0.8,1.5) {$V_i$};
\fill[color=black!30,rotate=-30] (-1,1) ellipse (0.2cm and 0.4cm) ;
\fill[color=black!30,rotate=30] (0.2,0.7) ellipse (0.1cm and 0.2cm);
\fill[color=black!30,rotate=30] (0.25,0.35) ellipse (0.03cm and 0.06cm);
\fill[color=black!30,rotate=30] (0.3,0.15) ellipse (0.03cm and 0.06cm);

\fill[color=black!30,rotate=-50] (0.55,1.8) ellipse (0.08cm and 0.18cm) ;
\fill[color=black!30,rotate=-80] (-0.2,1.8) ellipse (0.04cm and 0.11cm) ;
\fill[color=black!30,rotate=-100] (-0.65,1.7) ellipse (0.03cm and 0.08cm) ;
\end{scope}

\begin{scope}[shift={(7,0)}]
\draw(0,0) circle (1cm); \node at (-0.9,-1) {$K_{i_1}$};
\draw(2.2,0) circle (1cm); \node at (3.2,-1) {$K_{i_2}$};
\draw[color=red, line width=1.5pt] (0.5,1)-- node[above] {$E'$}(1.5,1);
\draw[color=red, line width=1.5pt] (0.5,-1)-- node[below] {$F'$}(1.5,-1);

\draw[dotted] (0.3,0) circle (0.6cm);
\draw[dotted] (0.4,0) circle (0.45cm);
\draw[dotted] (0.5,0) circle (0.3cm);
\draw[dotted] (0.55,0) circle (0.2cm);
\fill (0.55,0) circle (1pt);

\draw[dotted, rotate=180, xshift=-2.2cm] (0.3,0) circle (0.6cm);
\draw[dotted, rotate=180, xshift=-2.2cm] (0.4,0) circle (0.45cm);
\draw[dotted, rotate=180, xshift=-2.2cm] (0.5,0) circle (0.3cm);
\draw[dotted, rotate=180, xshift=-2.2cm] (0.55,0) circle (0.2cm);
\fill[rotate=180, xshift=-2.2cm] (0.55,0) circle (1pt);

\node at (-1,1.5) {$K_i$};
\fill[color=black!30,rotate=-30] (-1,1) circle (0.3cm) ;
\fill[color=black!30,rotate=30] (0.2,0.7) circle (0.15cm);
\fill[color=black!30,rotate=30] (0.25,0.35) circle (0.045cm);
\fill[color=black!30,rotate=30] (0.3,0.15) circle (0.045cm);

\fill[color=black!30,rotate=-50] (0.55,1.8) circle (0.1cm);
\fill[color=black!30,rotate=-80] (-0.2,1.8) circle (0.05cm);
\fill[color=black!30,rotate=-100] (-0.65,1.7) circle (0.05cm);
\end{scope}

\draw[->] (4,0)-- node[above]{$g$}(5,0);
\end{tikzpicture}
\caption{Extending the map $g$ by reflections in $V_{i_1},V_{i_2}$ and $K_{i_1},K_{i_2}$. Reflections eliminate the problems caused by these four large disks.}\label{figure:extension}
\end{figure}

If we now consider the images $E'$ and $F'$ with respect to the new domain obtained by reflections, we see that all bridges between $E'$ and $F'$ have been eliminated. This is because, first, no new bridges have been created by reflections as discussed above, and, second, the orbit of the problematic disks $K_{i_1},K_{i_2}$ with respect to reflections accumulates at the limit set of the corresponding Schottky group, which consists of two points; see Figure \ref{figure:extension}. Since bridges have been eliminated, we obtain favorable upper modulus bounds for the family $\Gamma(E',F';\widehat \C)$ with respect to the new domain. Namely, if $\Delta(E',F')$ is large, then the transboundary modulus of $\Gamma(E',F';\widehat \C)$ must be small (Proposition \ref{prop:modulus_upper}). 

Finally, since the extension of $g$ is quasiconformal in the new domain, it quasi\-pre\-serves transboundary modulus, and thus the transboundary modulus of $\Gamma(E',F';\widehat \C)$ is large. Therefore, we conclude that $\Delta (E',F')$ must be small. This completes the sketch of the proof.

\subsection*{Acknowledgments}
The author would like to thank Fangming Cai, Wen-Bo Li, Alex Rodriguez, and two anonymous referees for their comments and corrections, which greatly improved the presentation. 

\section{Preliminaries}
\subsection{Notation}
If $a$ is a parameter we use the notation $C(a),L(a)$, etc.\ for positive constants that depend only on the parameter $a$. We say that a statement is \textit{quantitative} if the constants or parameters appearing in the conclusions depend only on the constants or parameters in the assumptions.

Throughout the paper we mostly use the spherical metric $\sigma$ on the Riemann sphere  $\widehat \C=\C\cup \{\infty\}$. For $z,w\in \widehat \C$ we have
$$\sigma(z,w)= \inf_{\gamma}\int_{\gamma} \frac{2|dz|}{1+|z|^2},$$
where the infimum is taken over all rectifiable curves $\gamma$ in $\widehat \C$ connecting $z$ and $w$. For the purpose of estimating the spherical metric, we will also use the chordal metric on $\widehat \C$ defined for $z\in \C$ by 
$$\chi(z,w)= \frac{2|z-w|}{\sqrt{1+|z|^2}\sqrt{1+|w|^2}}  \,\,\, \text{when $w\in \C$ and}\,\,\, \chi(z,\infty)= \frac{2}{\sqrt{1+|z|^2}}.$$
Note that $\chi\leq \sigma\leq \frac{\pi}{2}\chi$. We will also use the spherical measure on $\widehat \C$, which is denoted by $\Sigma$ and is given by 
$$\Sigma(E)= \int_E \frac{4\, dm_2(z)}{(1+|z|^2)^2}$$
for a measurable set $E\subset \C$, where $m_2$ denotes the $2$-dimensional Lebesgue measure. The Euclidean area of a measurable set $E\subset \C$ is also denoted by $m_2(E)=|E|$.

Let $(X,d)$ be a metric space. The open ball of radius $r>0$ centered at a point $x\in X$ is denoted by $B_d(x,r)$ and the closed ball is denoted by $\bar B_d(x,r)$. The diameter of a set $E\subset X$ is denoted by $\diam_d(E)$. For the Euclidean metric in $\C$ we use the subscript $e$ when necessary. For example $B_e(z,r)=\{w\in \C: |z-w|<r\}$. Most of our considerations will be on the sphere $(\widehat \C,\sigma)$ and we will drop the subscript $\sigma$ from the notation. For example, if $E\subset \widehat \C$, then $\diam(E)$ denotes the spherical diameter of $E$. Also, when $z\in \widehat \C$ and $0<r<\pi$, $B(z,r)$ is called a (geometric) disk because its boundary is a circle. A \textit{region} or a \textit{domain} is a connected open subset of $\widehat \C$.

The Hausdorff distance of two sets $E,F$ in a metric space is the infimum of all $r>0$ such that $E$ is contained in the open $r$-neighborhood of $F$ and $F$ is contained in the $r$-neighborhood of $E$.

\subsection{Quasiconformal maps}

Let $K\geq 1$. An orientation-preserving homeomorphism $f\colon U\to V$ between open subsets of $\C$ is \textit{quasiconformal} if $f$ lies in the Sobolev space $W^{1,2}_{\loc}(U)$ and
$$\|Df(z)\|^2\leq KJ_f(z)$$ 
for a.e.\ $z\in U$; here $\|\cdot \|$ denotes the operator norm of the matrix $Df$ and $J_f$ is its Jacobian determinant. In this case we say that $f$ is $K$-quasiconformal. A homeomorphism $f\colon U\to V$ between open subsets of $\widehat \C$ is quasiconformal if $f|_{U\setminus \{\infty, f^{-1}(\infty)\}}$ is quasiconformal in the above sense. We will freely use the fundamental facts that the inverse of a quasiconformal map is quasiconformal, that the composition of two quasiconformal maps is quasiconformal, and that sets of measure zero are mapped to sets of measure zero under a quasiconformal map.  
We direct the reader to \cite{LehtoVirtanen:quasiconformal} for further background.

\begin{lemma}\label{lemma:conformal_composition}
Let $f\colon U\to V$, $g\colon V\to W$ be quasiconformal homeomorphisms between open subsets of $\widehat \C$ such that $f$ is $1$-quasiconformal on a measurable set $S\subset U$ and $g$ is $1$-quasiconformal on $f(S)$. Then $f^{-1}$ is $1$-quasiconformal on $f(S)$ and $g\circ f$ is $1$-quasiconformal on $S$.
\end{lemma}

Recall that a quasiconformal map $f$ is $1$-quasiconformal on a measurable set $S$ if $\|Df(z)\|^2=J_f(z)$ for a.e.\ $z\in S$. Since $f$ is quasiconformal, we have $J_f(z)>0$ for a.e.\ $z\in S$ \cite{LehtoVirtanen:quasiconformal}*{Section IV.1.5}. Thus, $f$ is $1$-quasiconformal on $S$ if and only if  $Df(z)$ is a conformal linear matrix for a.e.\ $z\in S$. A \textit{conformal linear matrix} is a positive scalar multiple of a special orthogonal matrix. Note that the inverse of a conformal linear matrix or the product of two such matrices is also a conformal linear matrix.

\begin{proof}
Since $f$ is quasiconformal, $f^{-1}$ is quasiconformal and we have $Df^{-1}( f(z))= Df(z)^{-1}$ for a.e.\ $z\in U$, as a consequence of the chain rule; see \cite{HenclKoskela:finitedistortion}*{Theorem 5.13} for a very general statement. The assumption that $f$ is $1$-quasiconformal on $S$ implies that $Df(z)$ is a conformal linear matrix for a.e.\ $z\in S$. Thus, $Df(z)^{-1}$ has the same property. Since $f$ maps sets of measure zero to sets of measure zero, we conclude that $Df^{-1}(w)$ is a conformal linear matrix for a.e.\ $w\in f(S)$, as desired. 

For the second assertion, by the chain rule we have $D(g\circ f)(z)= Dg(f(z))\cdot Df(z)$ for a.e.\ $z\in U$. By assumption, $Dg(f(z))$ and $Df(z)$ are conformal linear matrices for a.e.\ $z\in S$; here one uses the fact that $f^{-1}$ maps sets of measure zero to sets of measure zero. Thus, $D(g\circ f)(z)$ is a conformal linear matrix for a.e.\ $z\in S$.
\end{proof}

\begin{remark}\label{remark:conformal_composition}
One can extend the definition of ($1$-)quasiconformal maps and allow them to be orientation-reversing; for example, this is the case for anti-M\"obius transformations. Then the conclusion of Lemma \ref{lemma:conformal_composition} remains valid for such maps. In particular, if $f$ is orientation-preserving as in Lemma \ref{lemma:conformal_composition} and $\phi,\psi$ are anti-M\"obius transformations, then $\psi \circ f\circ \phi^{-1}$ is $1$-quasiconformal on $\phi(S)$.
\end{remark}

\begin{lemma}\label{lemma:lsc}
Let $U,V\subset \C$ be open sets and let $K\geq 1$. Let $f_n\colon U\to V$, $n\in \N$, be a sequence of $K$-quasiconformal homeomorphisms that converge locally uniformly to a quasiconformal map $f\colon U\to V$. If $f_n$ is $1$-quasiconformal on a measurable set $S\subset U$ for each $n\in \N$, then $f$ is also $1$-quasiconformal on $S$.
\end{lemma}

\begin{proof}
Let $B\subset U$ be a ball such that $\bar B\subset U$. By uniform convergence, $f_n(B)$ is contained in a fixed compact subset of $V$ for all $n\in\N$. By the change of variables formula \cite{LehtoVirtanen:quasiconformal}*{Section III.3.4} we have 
$$\int_B \|Df_n\|^2  \leq K \int_B J_{f_n}= K|f_n(B)|,$$
which is uniformly bounded in $n\in \N$. Thus, the sequence $\{f_n\}_{n\in \N}$ is uniformly bounded in $W^{1,2}(B)$. By the Banach--Alaoglu theorem, there exists a subsequence of $\{f_n\}_{n\in \N}$ that converges weakly in $W^{1,2}(B)$ to an element of $W^{1,2}(B)$. Since $\{f_n\}_{n\in \N}$ converges to $f$ strongly in $L^2(B)$, we see that $\{f_n\}_{n\in \N}$ converges to $f$ weakly in $W^{1,2}(B)$; see the discussion in \cite{Rickman:quasiregular}*{Section VI.7.8}. 

We now resort to \cite{GehringIwaniec:limit_finite_distortion}*{Theorem 6.1}, which implies that if $\|Df_n(z)\|^2 \leq M(z) J_{f_n}(z)$ for a.e.\ $z\in B$ and for some measurable function $M<\infty$,  and $\{f_n\}_{n\in \N}$ converges to $f$ weakly in $W^{1,2}(B)$, then $\|Df(z)\|^2\leq M(z) J_f(z)$ for a.e.\ $z\in B$. This gives immediately that for a.e.\ $z\in B\cap S$ we have $\|Df(z)\|^2= J_f(z)$. This is true for every ball $B\subset \bar B\subset U $, so we obtain $\|Df(z)\|^2=J_f(z)$ for a.e.\ $z\in S$.
\end{proof}

\subsection{Quasi-M\"obius and quasisymmetric maps}

Let $(X,d_X)$ be a metric space and $a,b,c,d\in X$ be distinct points. We define their (metric) cross ratio to be
$$[a,b,c,d]= \frac{d_X(a,c)d_X(b,d)}{d_X(a,d)d_X(b,c)}.$$
Observe that if $X=\widehat \C$, equipped with the chordal metric, and $a,b,c,d\neq \infty$, then 
$$[a,b,c,d]= \frac{|a-c||b-d|}{|a-d||b-c|}.$$
The factors containing the point $\infty$ are omitted. For example, 
$$[a,b,c,\infty]= \frac{|a-c|}{|b-c|}.$$

A homeomorphism $\eta\colon [0,\infty)\to [0,\infty)$ is called a \textit{distortion function}. A homeomorphism $f\colon X\to Y$ between metric spaces is \textit{quasi-M\"obius} if there exists a distortion function $\eta$ such that 
$$[f(a),f(b),f(c),f(d)] \leq \eta ([a,b,c,d])$$
for every quadruple of distinct points $a,b,c,d\in X$. In this case, we say that $f$ is $\eta$-quasi-M\"obius. Suppose that $X,Y\subset \widehat \C$ and $f\colon X\to Y$ is $\eta$-quasi-M\"obius. Then the cross ratios in $X$ and $Y$ may be computed with the chordal or the Euclidean distance interchangeably. Thus, we do not need to specify which of the two metrics is used in $X$ and $Y$. Also, since cross ratios are invariant under M\"obius transformations, the composition of $f$ with a M\"obius transformation is $\eta$-quasi-M\"obius as well.  Let $[a,b,c,d]_{\sigma}$ be the cross ratio computed with respect to the spherical metric. We have
$$\frac{4}{\pi^2}[a,b,c,d]\leq [a,b,c,d]_{\sigma} \leq \frac{\pi^2}{4}[a,b,c,d].$$
This shows that we have the freedom of switching between the Euclidean, chordal, or spherical metrics when dealing with quasi-M\"obius maps between subsets of $\widehat \C$ at the cost of changing slightly the distortion function.

A homeomorphism $f\colon X\to Y$ between metric spaces is \textit{quasisymmetric} if there exists a distortion function $\eta$ such that
$$ \frac{d_Y(f(a),f(b))}{d_Y(f(a),f(c))} \leq \eta\left(\frac{d_X(a,b)}{d_X(a,c)}\right)$$
for every triple of distinct points $a,b,c\in X$. In this case we say that $f$ is $\eta$-quasisymmetric.  

We list some fundamental properties of the above types of maps. 

\begin{enumerate}[label=\normalfont (Q-\arabic*)]
	\item\label{q:inverse} If $f$ is $\eta$-quasisymmetric (resp.\ quasi-M\"obius), then $f^{-1}$ is $\eta'$-quasi\-sym\-metric (resp.\ quasi-M\"obius) for $\eta'(t)=1/\eta^{-1}(t^{-1})$, $t>0$. Also, if $f$ is $\eta_1$-quasisymmetric (resp.\ quasi-M\"obius) and $g$ is $\eta_2$-quasisymmetric (resp.\ quasi-M\"obius), then $g\circ f$ is $\eta$-quasisymmetric (resp.\ quasi-M\"obius) for $\eta=\eta_2\circ \eta_1$.
	\item\label{q:qs_qm} Quasisymmetric maps are quasi-M\"obius, quantitatively \cite{Vaisala:quasimobius}*{Theorem 3.2}. Conversely, quasi-M\"obius maps between bounded metric spaces are quasisymmetric \cite{Vaisala:quasimobius}*{Theorem 3.12}; this statement is not quantitative. 
	\item\label{q:qm_qc} Quasi-M\"obius maps between open subsets of $\widehat \C$ are quasiconformal, quantitatively \cite{Vaisala:quasimobius}*{Theorem 5.2}. Conversely, quasiconformal homeomorphisms of the entire sphere $\widehat \C$ are quasi-M\"obius, quantitatively \cite{Vaisala:quasimobius}*{Theorem 5.4}.
	\end{enumerate}

A \textit{continuum} $E$ is a compact and connected metric space. If $E$ contains more than one point, we say that $E$ is a \textit{non-degenerate continuum}. We define the \textit{relative distance} of two disjoint non-degenerate continua $E,F$ in a metric space $(X,d)$ as
$$\Delta(E,F)=\frac{\dist_d(E,F)}{\min\{\diam_d(E),\diam_d(F)\}}.$$

\begin{lemma}\label{lemma:quasimobius:cross}
Let $f\colon (X,d_X)\to (Y,d_Y)$ be an $\eta$-quasi-M\"obius homeomorphism between metric spaces. Then there exists a distortion function $\widetilde \eta$ that depends only on $\eta$ such that for every pair of disjoint non-degenerate continua $E,F\subset X$ we have
$$\Delta(f(E),f(F)) \leq \widetilde \eta ( \Delta(E,F)).$$
\end{lemma}

\begin{proof}
For distinct points $x_1,x_2,x_3,x_4\in X$, we define
$$\langle x_1,x_2,x_3,x_4\rangle = \frac{\min\{d_X(x_1,x_3), d_X(x_2,x_4)\}}{\min\{d_X(x_1,x_4), d_X(x_2,x_3)\}}.$$
It is shown in \cite{Bonk:uniformization}*{Lemma 4.5} that for the functions $\eta_1(t)=3^{-1}\min\{t,\sqrt{t}\}$ and $\eta_2(t)=3\max\{t,\sqrt{t}\}$ we have
\begin{align}\label{lemma:quasimobius:cross:eta12}
\eta_1([x_1,x_2,x_3,x_4])\leq \langle x_1,x_2,x_3,x_4\rangle \leq \eta_2([x_1,x_2,x_3,x_4]).
\end{align}
We define $D(E,F)=\inf\{\langle x_1,x_2,x_3,x_4\rangle: x_1,x_4\in E\,\, \text{and}\,\, x_2,x_3\in F\}$. For the sake of brevity, we write $D(E,F)=\inf \langle x_1,x_2,x_3,x_4\rangle$. It is shown in \cite{Bonk:uniformization}*{Lemma 4.6} that 
\begin{align}\label{lemma:quasimobius:cross:delta}
\Delta(E,F)\leq D(E,F)\leq 2\Delta(E,F).
\end{align}
We use the notation $x'=f(x)$. Combining \eqref{lemma:quasimobius:cross:delta} and \eqref{lemma:quasimobius:cross:eta12}, we obtain
\begin{align*}
\Delta(f(E),f(F))&\leq D(f(E),f(F)) =\inf \langle x_1',x_2',x_3',x_4'\rangle\\
&\leq \inf \eta_2([x_1',x_2',x_3',x_4'])\\
&\leq \inf \eta_2\circ \eta ([x_1,x_2,x_3,x_4])\\
&\leq \inf \eta_2\circ \eta\circ \eta_1^{-1}( \langle x_1,x_2,x_3,x_4\rangle)\\
&=\eta_2\circ \eta\circ \eta_1^{-1} (\inf \langle x_1,x_2,x_3,x_4\rangle)\\
&\leq \eta_2\circ \eta\circ \eta_1^{-1} (2\Delta(E,F)).
\end{align*}
We obtain the desired conclusion for $\widetilde \eta(t)=\eta_2\circ \eta\circ \eta_1^{-1} (2t)$. 
\end{proof}

\subsection{Quasidisks and quasicircles}
A set $U\subset \widehat \C$ is a \textit{quasidisk} if it is the image of the unit disk $\D$ (or equivalently of any other disk) under a quasiconformal homeomorphism $f\colon \widehat \C\to \widehat \C$. If $f$ is $K$-quasiconformal for some $K\geq1$, we say that $U$ is a $K$-quasidisk. Note that the complement of a quasidisk is also a quasidisk.

Let $J\subset \widehat \C$ be a Jordan curve. We say that $J$ is a \textit{quasicircle} if there exists a constant $L\geq 1$ such that for every pair of points $x,y\in J$ there exists an arc $E\subset J$ with endpoints $x,y$ such that 
\begin{align}\label{definition:quasicircle}
\diam (E)\leq L\sigma(x,y).
\end{align}
In this case we say that $J$ is an $L$-quasicircle. We note that a subset of $\C$ is a quasicircle in the above sense if and only if \eqref{definition:quasicircle} holds with the Euclidean metric and with a possibly different constant $L'$. This equivalence is quantitative and follows from \cite{Bonk:uniformization}*{Prop.\ 4.4 (iii), (iv)}. The following deep result of Ahlfors provides the link between quasidisks and
quasicircles.  

\begin{theorem}\label{theorem:quasidisk_quasicircle}
Let $U\subset \widehat \C$ be a Jordan region and let $J=\partial U$. Then $U$ is a quasidisk if and only if $J$ is a quasicircle, quantitatively.
\end{theorem}

A proof can be found in \cite{Ahlfors:reflections}, \cite{LehtoVirtanen:quasiconformal}*{Theorem II.8.6} or \cite{Pommerenke:conformal}*{Section 5}. A consequence of this result and of the fact that quasiconformal maps preserve sets of measure zero is that quasicircles have measure zero.

\begin{proposition}\label{prop:quasidisk_extension}
Let $\Omega \subset \widehat \C$ be a domain such that $\widehat \C\setminus \bar \Omega$ is the union of finitely many quasidisks with disjoint closures. Then every quasi-M\"obius embedding $f\colon \Omega\to \widehat \C$ has an extension to a quasiconformal homeomorphism of $\widehat \C$, quantitatively.
\end{proposition}

This follows immediately from the more general result \cite{Bonk:uniformization}*{Prop.\ 5.1}. Alternatively, a simple proof of Proposition \ref{prop:quasidisk_extension} can be given based on the Beurling--Ahlfors extension \cite{BeurlingAhlfors:extension} and on the quasiconformal removability of quasicircles. 

\section{Circularizable pairs of Jordan regions}

\subsection{Elementary properties}

Recall that a pair of Jordan regions $U,V$ in $\widehat \C$ with $\partial U\neq \partial V$ is $K$-{quasiconformally circularizable} for some $K\geq 1$ if there exists a $K$-quasiconformal homeomorphism $\phi\colon \widehat \C\to \widehat \C$ such that $\phi(U)$ and $\phi(V)$ are geometric disks. In this case, if $f$ is a $K'$-quasiconformal homeomorphism of $\widehat \C$, then $f(U)$ and $f(V)$ are $KK'$-quasiconformally circularizable.

The next proposition implies that the main assumption of Bonk \cite{Bonk:uniformization} on uniformly relatively separated uniform quasicircles is stronger than our assumption in Theorem \ref{theorem:main} on quasiconformally pairwise circularizable Jordan regions.

\begin{proposition}\label{proposition:urs}
Let $L\geq 1$, $\delta>0$ and let $U,V\subset \widehat \C$ be a pair of $L$-quasidisks with disjoint closures such that 
$\Delta(\partial U, \partial V)\geq \delta$.
Then the pair $U,V$ is $K(L,\delta)$-quasicon\-for\-mal\-ly circularizable.
\end{proposition}

When $\infty \in U$ and if the Euclidean metric is used, then the conclusion of Proposition \ref{proposition:urs} was first observed by Herron \cite{Herron:uniform}*{Theorem 2.6 and Corollary 3.5}. As a consequence of Lemma \ref{lemma:quasimobius:cross}, the assumptions of the above proposition can be equivalently stated using the Euclidean metric, in a quantitative fashion.

Observe that if $U,V$ are quasiconformally (or merely homeomorphically) circularizable Jordan regions, then $\partial U\cap \partial V$ contains at most one point. We prove another important consequence of quasiconformal circularizability.

\begin{lemma}[No room for three]\label{lemma:three_regions}
Suppose that a pair of disjoint Jordan regions $U,V\subset \widehat \C$ is quasiconformally circularizable. Let $W\subset \widehat \C\setminus (U\cup V)$ be a Jordan region such that $\partial W\cap \partial U\cap \partial V\neq \emptyset$. Then $W$ is not a quasidisk. 
\end{lemma}

The proof is elementary and relies on the characterization of Ahlfors from Theorem \ref{theorem:quasidisk_quasicircle}. Namely, if $U,V$ are disks tangent to each other and $W$ is a Jordan region as in the statement, then $W$ has a cusp at the point of tangency of $U$ and $V$, so it is not a quasidisk. We now include the technical details.

\begin{proof}
By assumption, there exists a quasiconformal map of $\widehat \C$ that maps $U$ and $V$ to the half-planes $U'=\{z\in \C: \im(z)< 0\}$ and $V'=\{z\in \C: \im(z)>1\}$, respectively. Let $W\subset \{z\in \C: 0<\im(z)<1\}$ be a Jordan region such that $\infty \in \partial W$. By the quasiconformal invariance of quasidisks, it suffices to show that $W$ is not a quasidisk. Let $z_0\in \partial W\setminus \{\infty\}$ be a point and consider two arcs $E_1,E_2$ of $\partial W$ connecting $z_0$ to $\infty$ that have disjoint interiors. By the connectedness of $E_1$, without loss of generality, we may assume that $\{z\in \C: \re(z)=t\}$ has non-empty intersection with $E_1$ for each $t\geq \re(z_0)$. 

If $E_2$ does not have the same property, then the connectedness of $E_2$ would imply that $\{z\in \C: \re(z)=t\}$ has non-empty intersection with $E_2$ for each $t\leq \re(z_0)$. We will show that $\partial W$ separates $U'$ and $V'$. If not, there exists $\varepsilon>0$ and a Jordan arc $\alpha$ lying in $\{z\in \C: -\varepsilon <\im(z)<1+\varepsilon\}\setminus \partial W$ except for its endpoints $z_1,z_2$, which satisfy $\im(z_1)=-\varepsilon$ and $\im(z_2)=1+\varepsilon$. By the Jordan curve theorem, the Jordan arc $\alpha$ and the lines $\im(z)=-\varepsilon$ and $\im(z)=1+\varepsilon$ define two disjoint Jordan regions whose boundaries pass through $\infty$. By the properties of $E_1$ and $E_2$, one of these Jordan regions intersects $E_1$ and the other intersects $E_2$. Thus the connected set $(E_1\cup E_2)\setminus \{\infty\}=\partial W\setminus \{\infty\}$ must intersect the common boundary $\alpha$ of these Jordan regions. This is a contradiction. Therefore, as claimed, $\partial W$ separates $U'$ and $V'$. The Jordan curve theorem implies that one of $U'$ and $V'$ is a subset of the Jordan region $W$, a contradiction. This contradiction implies that $\{z\in \C: \re(z)=t\}$ has non-empty intersection with both $E_1$ and $E_2$ for each $t\geq \re(z_0)$.

Let $z_1,z_2\in \partial W$ be distinct points with $t=\re(z_1)=\re(z_2)> \re(z_0)$. Then 
$$\sigma(z_1,z_2)\leq \frac{\pi}{2} \chi(z_1,z_2) \leq  \frac{\pi}{1+t^2}.$$
Let $E$ be the arc of $\partial W\setminus \{z_1,z_2\}$ that contains $\infty$. Then 
$$\diam (E) \geq \sigma(z_1,\infty) \geq \chi(z_1,\infty)=\frac{2}{\sqrt{1+|z_1|^2}} \geq \frac{2}{\sqrt{1+(|t|+1)^2}}.$$
Now, if $F$ is the complementary arc of $E$, then 
$$\diam(F) \geq \sigma(z_1,z_0) \geq \chi (z_1,z_0) \geq \frac{2(t-\re(z_0))}{\sqrt{1+(|t|+1)^2} \sqrt{1+|z_0|^2}}.$$
As $t\to \infty$, we see that the ratio $\min\{\diam(E),\diam(F)\}/ \sigma(z_1,z_2)$ tends to $\infty$, so $\partial W$ is not a quasicircle. By Theorem \ref{theorem:quasidisk_quasicircle}, $W$ is not a quasidisk. 
\end{proof}

\subsection{Pulling away Jordan regions}
We prove a statement that allows us to reduce Theorem \ref{theorem:main} to the case of Jordan regions with disjoint closures.

\begin{proposition}[Exhaustion]\label{prop:disjoint_closures}
Let $\{U_i\}_{i\in I}$ be a finite collection of disjoint Jordan regions in $\widehat \C$ that are $K$-quasiconformally pairwise circularizable for some $K\geq 1$. Then for each $n\in \N$ there exists a collection $\{U_{i}(n)\}_{i\in I}$ of Jordan regions 
with the following properties.
\begin{enumerate}[label=\normalfont(\roman*)]
\item\label{prop:disjoint_closures:1} $U_i(n)\subset  U_i(n+1)\subset U_i$ and $\bigcup_{n\in \N} U_i(n)= U_i$ for each $i\in I$. 
\item\label{prop:disjoint_closures:3} The regions $\{U_{i}(n)\}_{i\in I}$ are $H(K)$-quasiconformally pairwise circularizable.
\item\label{prop:disjoint_closures:4} The closures $\{\bar{U_{i}(n)}\}_{i\in I}$ are disjoint.
\end{enumerate}
\end{proposition}

Observe that Theorem \ref{theorem:main}, once proved, implies a stronger version of the above proposition for countably many (rather than finitely many) Jordan regions $U_i$, $i\in I$. 

For the proof we need two elementary lemmas that provide some useful quasiconformal transformations of the plane. We thank an anonymous referee for suggesting alternative proofs of these lemmas, which substantially simplify the original arguments.

\begin{lemma}[Quasiconformal pushing of a chord onto a circular arc]\label{lemma:qcpush}
Let $\theta\in (0,\pi/4)$ and consider the disk $D=\{z\in \C: |z-\sec\theta|< \tan\theta\}$ whose boundary is orthogonal to the unit circle at the points $e^{-i\theta}$, $e^{i\theta}$. There exists a $2$-quasi\-con\-formal map $f$ of $\C$ that is the identity map in $\C\setminus D$ and maps the chord $C=[e^{-i\theta},e^{i\theta}]$ of the unit circle onto the minor arc of the unit circle subtended by $C$.

In particular, if $T$ is the component of $\D\setminus [e^{-i\theta},e^{i\theta}]$ that contains the origin $0$, then $f$ maps $T$ onto $\D$. 
\end{lemma}

In the proof we use the elementary fact that the map $g(re^{it})=re^{i\alpha t}$, which stretches or contracts angles by a factor $\alpha>0$, is $K$-quasiconformal (in an appropriate region) for $K=\max\{\alpha,1/\alpha\}$.

\begin{figure}
\begin{tikzpicture}[scale=2]
 	\begin{scope}
	\coordinate (a) at (-30:1);
	\coordinate (b) at (30:1);	
	\draw[dotted] (0,0) circle (1cm);
	\node at (-0.7,0) {$\D$};
 	\fill (0,0) circle (0.7pt) node[above]{$0$};
	\draw[red,line width=1.5pt] (a)--(b);
	\fill (a) circle (0.7pt);
	\fill (b) circle (0.7pt)node[above] {$e^{i\theta}$};
	\draw (1.15,0) circle (0.577cm);
	\draw[dotted](0,0)--(1.15,0);
	\draw[dotted](0,0)--(a);
	\draw ([shift=(0:0.2)]0,0) arc (0:-30:0.2) node[shift={(0.2,0.05)}]{$\scriptstyle \theta$};
	\fill (1.15,0) circle (0.7pt) node[above] {$\scriptstyle \sec\theta$};
	\draw node at (1.5,0) {$D$};
	\draw[dotted] (1.15,0)--(a);
	\draw ([shift=(60:0.2)]a) arc (60:90:0.2) node[shift={(0.12,0.13)}]{$\scriptstyle \theta$};
	\end{scope}
	
	\begin{scope}[shift={(4,0)}]
	\draw[dotted] (-1,0)--(1,0);
	\draw (0,-1)--(0,1)node[right] {$M(\partial D)$};
	\draw[red,line width =1.5pt] (0,0)--(1,0.575) node[above]{$\color{black}M(C)$};
	\draw ([shift=(0:0.2)]0,0) arc (0:30:0.2) node[shift={(0.2,-0.05)}]{$\scriptstyle \theta$};
	\fill (0,0) circle (0.7pt) node[anchor=north east]{$0$};
	\end{scope}
	
	\begin{scope}[shift={(4,-3)}]
	\draw[dotted] (-1,0)--(1,0);
	\draw (0,-1)--(0,1)node[right] {$M(\partial D)$};
	\draw[red,line width=1.5pt] (0,0)--(1,0);
	\fill (0,0) circle (0.7pt) node[anchor=north east]{$0$};
	\end{scope}
	
	\begin{scope}[shift={(0,-3)}]
	\coordinate (a) at (-30:1);
	\coordinate (b) at (30:1);	
	\draw[dotted] (0,0) circle (1cm);
 	\fill (0,0) circle (0.7pt) node[above]{$0$};
	\draw (1.15,0) circle (0.577cm);
	\draw[red, line width=1.5pt] ([shift=(-30:1)]0,0) arc (-30:30:1);
	\fill (a) circle (0.7pt);
	\fill (b) circle (0.7pt)node[above] {$e^{i\theta}$};
	\end{scope}
	
	\draw[->] (2,0)--(2.5,0) node[pos=.5,above]{$M$};
	\draw[<-] (2,-3)--(2.5,-3) node[pos=.5,above]{$M^{-1}$};
	\draw[->] (4,-1.25)--(4,-1.75) node[pos=.5,right]{$g$};
	\draw[->] (0,-1.25)--(0,-1.75) node[pos=.5,right]{$f$};
\end{tikzpicture}
	\caption{The construction of the map $f$ in the proof of Lemma \ref{lemma:qcpush}.}\label{figure:qcpush}
\end{figure}

\begin{proof}
Consider the M\"obius transformation $M(z)= \frac{1-e^{i\theta}z}{z-e^{i\theta}}$, which maps the unit disk onto the upper half-plane so that $M(e^{-i\theta})=0$, $M(1)=1$, and $M(e^{i\theta})=\infty$. The disk $D$ is mapped onto the right half-plane and the chord $C$ is mapped onto the ray $\{re^{i\theta}: r\geq 0\}\cup \{\infty\}$; this can be derived from the fact that the angle between the chord $C$ and the unit circle is $\theta$. See Figure \ref{figure:qcpush} for an illustration. We define a homeomorphism $g$ of $\C$ that is equal to the identity map in the left half-plane, while it stretches and contracts angles appropriately in the right half-plane so that $M(C)$ is mapped onto the non-negative real axis. Specifically,
\begin{align*}
g(re^{it})=
\begin{cases}
r \exp\big(i \frac{\pi}{2} \frac{t-\theta}{\pi/2-\theta}\big),& \theta\leq t\leq \frac{\pi}{2}\\
r \exp\big(i \frac{\pi}{2} \frac{t-\theta}{\pi/2+\theta}\big),& -\frac{\pi}{2}\leq t\leq \theta
\end{cases}. 
\end{align*}
The map $g$ is $L_1$-quasiconformal for $L_1=\frac{\pi/2}{\pi/2-\theta}$ when $\theta<t<\frac{\pi}{2}$ and $L_2$-quasiconformal for $L_2=\frac{\pi/2+\theta}{\pi/2}$ when $-\frac{\pi}{2}<t<\theta$. Also, $g$ is trivially $1$-quasi\-conformal in the left half-plane. Note that $L_2<L_1\leq 2$ for $\theta\in (0,\pi/4)$. The quasiconformal removability of lines (see \cite{LehtoVirtanen:quasiconformal}*{Theorem I.8.3}) implies that $g$ is $2$-quasiconformal in $\C$. The composition $f=M^{-1}\circ g\circ M$ is $2$-quasiconformal and has the desired behavior.
\end{proof}

\begin{lemma}[Quasiconformal detachment of a disk from a half-plane]\label{lemma:qcpull}
Let $\theta\in (0,\pi/2)$ and $T$ be the component of $\D\setminus [e^{-i\theta},e^{i\theta}]$ that contains the origin $0$.  There exists a $2$-quasiconformal map of $\C$ that is the identity map in the half-plane $\{z\in \C: \re(z)\geq 1\}$ and maps $T$ onto a disk. 
\end{lemma}

\begin{figure}
\begin{tikzpicture}[scale=2]
 	\begin{scope}
	\coordinate (a) at (-30:1);
	\coordinate (b) at (30:1);
	
		\begin{scope}
		\clip (-1,-1) rectangle (0.866,1);
		\fill[yellow!60] (0,0) circle (1cm);
		\end{scope}
		
	\draw[dotted] (0,0) circle (1cm);
	\node at (-0.7,0) {$T$};
 	\fill (0,0) circle (0.7pt) node[above]{$0$};
	\draw[red,line width=1.5pt] (a)--(b);
	\draw[red,line width=1pt,dashed] (0.866,-1.2)--(0.866,1.2);
	\fill (a) circle (0.7pt);
	\fill (b) circle (0.7pt)node[left] {$e^{i\theta}$};
	\draw ([shift=(63:0.2)]a) arc (63:90:0.2) node[shift={(0.12,0.13)}]{$\scriptstyle \theta$};
	\fill[black!20] (1,-1.2)rectangle(1.5,1.2);
	\draw (1,-1.2)--(1,1.2);
	\fill (1,0) circle (.7pt) node[right]{$1$};
	\node at (1.25,-0.5) {$\Omega$};
	\end{scope}
	
	\begin{scope}[shift={(4,0)}]
	\fill[yellow!60] (-1,0)--(0,0)--(1,1.2)--(-1,1.2);
	\node at (-0.5,0.5) {$M(T)$};
	\draw[dotted] (-1,0)--(1,0);
	\draw[red,line width =1.5pt] (0,0)--(1,1.2);
	\draw[red, line width=1pt, dashed] (0,0)--(-1,-1.2);
	\draw ([shift=(0:0.2)]0,0) arc (0:50:0.2) node[shift={(0.2,-0.05)}]{$\scriptstyle \theta$};
	\fill (0,0) circle (.7pt) node[above]{$0$};
	\fill[black!20] (0.4,-0.866) circle (0.866cm);
	\draw (0.4,-0.866) circle (0.866cm) node {$M(\Omega)$};
	\fill (0.4,0) circle (.7pt) node[below]{$1$};
	\end{scope}
	
	\begin{scope}[shift={(4,-3)}]
	\fill[yellow!60] (-0.9,-0.422)--(0.9,0.422)--(-0.9,0.422);
	\draw[dotted] (0,0)--(1,0);
	\draw[dotted] (0,0)--(-0.9,-0.422);
	\draw[red,line width =1.5pt] (0,0)--(0.9,0.422);
	\draw[red, line width=1pt, dashed] (0,0)--(-1,-1.2);
	\draw ([shift=(0:0.2)]0,0) arc (0:25:0.2);
	\fill (0,0) circle (.7pt) node[above]{$0$};
	\fill[black!20] (0.4,-0.866) circle (0.866cm);
	\draw (0.4,-0.866) circle (0.866cm);
	\fill (0.4,0) circle (.7pt) node[below]{$1$};
	\end{scope}
	
	\begin{scope}[shift={(0,-3)}]
	\coordinate (a) at (-30:1);
	\coordinate (b) at (30:1);	
	\draw[dotted] ([shift=(-30:1)]0,0) arc (-30:30:1);
 	\fill (0,0) circle (.7pt) node[above]{$0$};
	
		\begin{scope}
		\clip (-1,-2) rectangle (1,2);
		\fill[yellow!60] (b) arc (20:380:1.5);
		\draw[dotted] (b) arc (20:340:1.5);
		\end{scope}
	
	\draw[red,line width=.7pt] (a) arc (-20:20:1.5);	
	\draw[red,line width=1pt,dashed] (0.866,-1.2)-- (a); 
	\draw[red,line width=1pt,dashed](0.866,1.2)--(b);
	\fill (a) circle (.7pt);
	\fill (b) circle (.7pt)node[left] {$e^{i\theta}$};
	\fill[black!20] (1,-1.2)rectangle(1.5,1.2);
	\draw (1,-1.2)--(1,1.2);
	\fill (1,0) circle (0.7pt) node[right]{$1$};
	\end{scope}
	
	\draw[->] (2,0)--(2.5,0) node[pos=.5,above]{$M$};
	\draw[<-] (2,-3.25)--(2.5,-3.25) node[pos=.5,above]{$M^{-1}$};
	\draw[->] (4,-1.9)--(4,-2.25) node[pos=.5,right]{$g$};
	\draw[->] (0,-1.1)--(0,-1.5) node[pos=.5,right]{$f$};
\end{tikzpicture}
	\caption{The construction of the map $f$ in the proof of Lemma \ref{lemma:qcpull}.}\label{figure:qcpull}
\end{figure}

\begin{proof}
As in the proof of Lemma \ref{lemma:qcpush}, consider the M\"obius transformation $M$ that maps the unit disk onto the upper half-plane and satisfies $M(e^{-i\theta})=0$, $M(1)=1$, and $M(e^{i\theta})=\infty$.  Note that the image of the half-plane $\Omega=\{z\in \C:\re(z)>1\}$ under $M$ is a disk in the lower half-plane that is tangent to the real line at the point $1$ and to the line $\{re^{i\theta}: r\in \R\}$. Thus, $M(\Omega)$ is contained in the region $\{re^{it}:r>0,\,\, \pi+\theta < t < 2\pi\}$. Also, observe that $M(T)=\{re^{it}: r>0,\,\, \theta<t<\pi\}$;  see Figure \ref{figure:qcpull}. We will define a homeomorphism $g$ of $\C$ that is the identity map in $M(\Omega)$ and maps $M(T)$ onto a half-plane.

For $r\geq 0$ and $\pi+\theta\leq t \leq 2\pi$ we set $g(re^{i\theta})=re^{i\theta}$, so $g$ is the identity map in $M(\Omega)$. For $0\leq t\leq \theta$ we set $g(re^{it})= re^{it/2}$ and symmetrically for $\pi\leq t\leq \pi+\theta$ we set $g(re^{it})=re^{i(t+\pi+\theta)/2}$. So in these two sectors the map $g$ contracts angles by the factor $2$. Finally, for $\theta\leq t\leq \pi$ we define $g$ by 
$$g(re^{it})=r \exp\left(i \bigg(\frac{\pi}{\pi-\theta} (t-\theta)+\frac{\theta}{2}\bigg)\right),$$
which stretches angles by the factor $\frac{\pi}{\pi-\theta}$, a quantity bounded above by $2$ for $\theta\in (0,\pi/2)$. By definition, $g$ maps $M(T)$ onto the half-plane $\{re^{it}:r>0,\,\, \theta/2<t<\pi+\theta/2\}$, as shown in Figure \ref{figure:qcpull}. Also, note that $g$ is $2$-quasiconformal by the quasiconformal removability of lines. The map $f=M^{-1}\circ g\circ M$ has the desired behavior.
\end{proof}

\begin{proof}[Proof of Proposition \ref{prop:disjoint_closures}]
Let $\{U_i\}_{i\in I}$ be a finite collection of disjoint Jordan regions that are $K$-quasiconformally pairwise circularizable. In particular, each $U_i$, $i\in I$, is a $K$-quasidisk. By Lemma \ref{lemma:three_regions}, no point of $\widehat \C$ can lie on the boundary of three regions $U_i$, $i\in I$. Also, the boundaries of any pair of regions have at most one point in common. For $i,j\in I$, $i\neq j$, let $f_{ij}$ be a $K$-quasiconformal homeomorphism of $\widehat \C $ that maps $U_i$ and $U_j$ to disks $D_{ij}$ and $D_{ji}$ in $\C$, respectively. We assume that $f_{ij}=f_{ji}$ for $i\neq j$. 

Since $I$ is a finite index set and any two sets $\partial U_i$ and $\partial U_j$, $i\neq j$, have at most one point in common, we see that the set $\bigcup_{i\neq j}(\partial U_i\cap \partial U_j)$ is finite. For each $p\in \partial U_i\cap \partial U_j$, $i\neq j$, there exists an arbitrarily small neighborhood $V(p)$ of $p$ such that the sets $V(p)\cap \partial U_i$ and $V(p)\cap \partial U_j$ are arcs. Moreover, by choosing small enough neighborhoods we may have that $\bar {V(p)}\cap \bar{V(q)}=\emptyset$ whenever $p\neq q$. We finally require that $\bar{V(p)}$ is disjoint from $\bar{U_k}$ if $p\notin \partial U_k$. Therefore, if $p\in \partial U_i\cap \partial U_j$, then $\bar{V(p)}$ intersects only $U_i$ and $U_j$ and no other region $\bar{U_k}$.

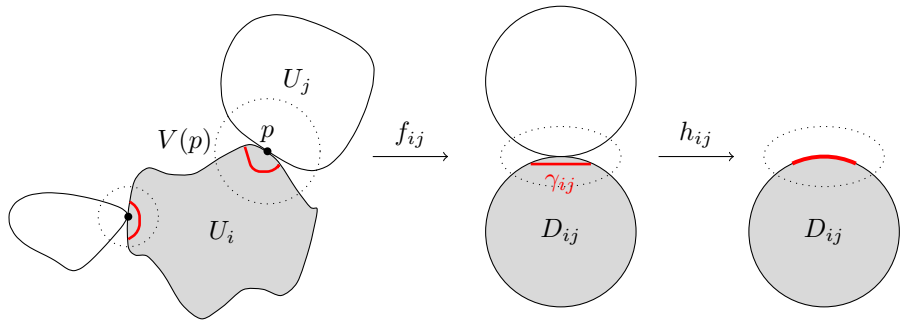
\begin{figure}

\begin{tikzpicture}

\begin{scope}           		
\def\a{1.25};\def\A{1};
\def\b{0.1};\def\B{2};
\def\c{1};\def\C{1};
\def\d{0.2};\def\D{5};
\draw [fill=black!15,smooth,domain=0:360] plot 
            		(
            		{\a*cos( \A*\x )+\b*sin(\B*\x)}, 
            		{\c*sin( \C*\x ) + \d*cos(\D*\x )}
            		) ;
            		
\def\a{1};\def\A{1};
\def\b{0.1};\def\B{2};
\def\c{1};\def\C{1};
\def\d{0.1};\def\D{3};
\draw [smooth,domain=0:360, shift={(1,1.85)}] plot 
            		(
            		{\a*cos( \A*\x )+\b*sin(\B*\x)}, 
            		{\c*sin( \C*\x ) + \d*cos(\D*\x )}
            		) ;
            		
\def\a{0.7};\def\A{1};
\def\b{0.2};\def\B{2};
\def\c{0.5};\def\C{1};
\def\d{0.05};\def\D{3};
\draw [smooth,domain=0:360, shift={(-2.05,0)}] plot 
            		(
            		{\a*cos( \A*\x )+\b*sin(\B*\x)}, 
            		{\c*sin( \C*\x ) + \d*cos(\D*\x )}
            		) ;            		

\fill (-1.25,0.2) node (q) {} circle (1.5pt);
\draw[dotted] (q) circle (0.4cm);
\draw[red, line width=1pt, rounded corners=3pt] (-1.23,0.4)--(-1.1,0.3)--(-1.1,0)--(-1.25,-0.1);

\fill (0.6,1.07) node[above] {$p$} circle (1.5pt);
\draw[dotted] (0.6,1.07) circle (0.7cm);
\node at (-0.5,1.2) {$V(p)$};
\draw[red, line width=1pt, rounded corners=3pt] (0.3,1.13)--(0.4,0.8)--(0.7,0.8)--(0.73,0.85);

\node at (0,0) {$U_i$};
\node at (1,2) {$U_j$};
\end{scope}

\draw[->] (2,1)-- node[above]{$f_{ij}$}(3,1);

\begin{scope}[shift={(4.5,0)}]
\draw[fill=black!15] (0,0) node {$D_{ij}$} circle (1cm);
\draw (0,2) circle (1cm);
\draw[red, line width =1pt] (-0.4,0.9)-- node[below]{$\gamma_{ij}$}(0.4,0.9);
\draw[dotted] (0,1) ellipse (0.8cm and 0.4cm);
\end{scope}

\draw[->] (5.8,1)-- node[above]{$h_{ij}$}(6.8,1);

\begin{scope}[shift={(8,0)}]
\draw[fill=black!15] (0,0) node{$D_{ij}$} circle (1cm);
\draw[red, line width =1.5pt] ([shift=(65:1)]0,0) arc (65:115:1);
\draw[dotted] (0,1) ellipse (0.8cm and 0.4cm);
\end{scope}
\end{tikzpicture}
\caption{The construction of the Jordan region $U_i'$, which is the part of the region $U_i$ bounded by the (red) curves $f_{ij}^{-1}(\gamma_{ij})$ and a portion of $\partial U_i$.}\label{figure:construction_ui}
\end{figure}

Let $p\in \partial U_i\cap \partial U_j$ and consider the disk $D_{ij}=f_{ij}(U_i)$, which is tangent to $D_{ji}=f_{ij}(U_j)$ at the point $f_{ij}(p)$. Consider a chord $\gamma_{ij}$ of $D_{ij}$ that is parallel to the tangent line at $f_{ij}(p)$, lies within $f_{ij}(V(p))$, and separates the center of the disk $D_{ij}$ from the point of tangency $f_{ij}(p)$. Let $D_{ij}'$ be the component of $D_{ij}\setminus \gamma_{ij}$ that contains the center of $D_{ij}$. Observe that $D_{ij}'$ has a positive distance from $D_{ji}$. Thus, the same is true for the preimages $f_{ij}^{-1}(D_{ij}')$ and $U_{j}$. See Figure \ref{figure:construction_ui} for an illustration.

By Lemma \ref{lemma:qcpush}, if $\gamma_{ij}$ is sufficiently close to $f_{ij}(p)$ then there exists a $2$-quasi\-conformal map $h_{ij}$ of $\C$ that is the identity map outside $f_{ij}(V(p))$ and maps $D_{ij}'$ onto $D_{ij}$.  (This map distorts the disk $D_{ji}$, but this does not concern us.) Note that $h_{ij}$ extends to a $2$-quasiconformal homeomorphism of $\widehat \C$. Thus, the map $f_{ij}^{-1}\circ h_{ij}\circ f_{ij}$ is the identity outside $V(p)$ and maps $V(p)\cap f_{ij}^{-1}(D_{ij}')$ onto $V(p)\cap U_i$ in a $2K^2$-quasiconformal fashion.

For each $i\in I$ we define $I(i)= \{j\in I\setminus \{i\}: \partial U_i\cap \partial U_j\neq \emptyset\}$.  Now we define $U_i'= \bigcap_{j\in I(i)}f_{ij}^{-1}(D_{ij}')$ whenever $I(i)\neq \emptyset$ and $U_i'=U_i$ otherwise. In the first case, $U_i'$ is obtained from $U_i$ by removing finitely many disjoint Jordan regions, each bounded by an arc $f_{ij}^{-1}(\gamma_{ij})$ and an arc of $\partial U_i$ with the same endpoints; it is important here that $V(p)\cap \partial U_i$ is an arc for each $p\in \partial U_i\cap \partial U_j$, $j\in I(i)$. Therefore, the set $U_i'$ is a Jordan region that is contained in $U_i$. Moreover, by construction we have $\bar{U_i'}\cap \bar{U_j'}=\emptyset$ for each $j\neq i$, as required in conclusion \ref{prop:disjoint_closures:4}; see Figure \ref{figure:construction_ui}. Since the above construction can be carried out so that the chord $\gamma_{ij}$ is arbitrarily close to the corresponding tangency point between $D_{ij}$ and $D_{ji}$, we see that the regions $U_i'$ can be taken to be arbitrarily close to $U_i$, and we can ensure that condition \ref{prop:disjoint_closures:1} is satisfied as well. 

It remains to show that the regions $\{U_i'\}_{i\in I}$ are $H(K)$-quasiconformally pairwise circularizable, as required in \ref{prop:disjoint_closures:3}. Let $i\in I$ and $k\notin I(i)$. Then $\bar {U_i}$ is disjoint from $\bar {U_k}$. Thus, $f_{ik}$ maps $U_i$ and $U_k$ to disks with disjoint closures. However, $f_{ik}(U_i')$ and $f_{ik}(U_k')$ are only subsets of these disks. Let $j\in I(i)$ be such that $\partial U_i\cap \partial U_j =\{p\}$. The map $f_{ik}\circ  f_{ij}^{-1}\circ h_{ij}\circ f_{ij}\circ f_{ik}^{-1}$ is the identity outside $f_{ik}(V(p))$ and maps $f_{ik}(V(p)\cap U_i')$ onto $f_{ik}(V(p)\cap U_i)$ in a $2K^4$-quasiconformal fashion. Also, $\bar{V(p)}\cap \bar{U_k}=\emptyset$ by the choices in the beginning of the proof. The analogous statements apply if $j\in I(k)$. Since the sets $V(p)$ have disjoint closures, we may paste these maps to construct a homeomorphism $\phi_{ik}$ of $\widehat \C$ that maps $f_{ik}(U_i')$ and $f_{ik}(U_k')$ onto $f_{ik}(U_i)$ and $f_{ik}(U_k)$, respectively; see Figure \ref{figure:construction_phi}. The composition $\phi_{ik}\circ f_{ik}$ is $2K^5$-quasiconformal and maps $U_i'$ and $U_k'$ to disks, as desired. 

\begin{figure}
\begin{tikzpicture}
\begin{scope}
\draw[fill=black!15] (0,0) node {$f_{ik}(U_i')$} circle (1cm);
\fill[white] ([shift=(120:1)]0,0) arc (-170:26:0.2cm);
\draw[red] ([shift=(120:1)]0,0) arc (-170:26:0.2cm);
\draw[dotted] ([shift=(110:1)]0,0) circle (0.4cm);
\fill[white] ([shift=(0:1)]0,0) arc (90:246:0.2cm);
\draw[red] ([shift=(0:1)]0,0) arc (90:250:0.2cm);
\draw[dotted] ([shift=(-10:1)]0,0) circle (0.4cm);
\end{scope}

\begin{scope}[shift={(0,2.5)}, rotate=50]
\draw[fill=black!15] (0,0) node {$f_{ik}(U_k')$} circle (1cm);
\fill[white] ([shift=(120:1)]0,0) arc (-170:26:0.2cm);
\draw[red] ([shift=(120:1)]0,0) arc (-170:26:0.2cm);
\draw[dotted] ([shift=(110:1)]0,0) circle (0.4cm);
\fill[white] ([shift=(0:1)]0,0) arc (90:246:0.2cm);
\draw[red] ([shift=(0:1)]0,0) arc (90:250:0.2cm);
\draw[dotted] ([shift=(-10:1)]0,0) circle (0.4cm);
\end{scope}

\begin{scope}[shift={(5,0)}]
\draw[fill=black!15] (0,0)  circle (1cm);
\draw[red, line width=1.5] ([shift=(95:1)]0,0) arc (95:120:1cm);
\draw[dotted] ([shift=(110:1)]0,0) circle (0.4cm);
\draw[red, line width=1.5] ([shift=(4:1)]0,0) arc (4:-25:1cm);
\draw[dotted] ([shift=(-10:1)]0,0) circle (0.4cm);
\end{scope}

\begin{scope}[shift={(5,2.5)}, rotate=50]
\draw[fill=black!15] (0,0)  circle (1cm);
\draw[red, line width=1.5] ([shift=(95:1)]0,0) arc (95:120:1cm);
\draw[dotted] ([shift=(110:1)]0,0) circle (0.4cm);
\draw[red, line width=1.5] ([shift=(4:1)]0,0) arc (4:-25:1cm);
\draw[dotted] ([shift=(-10:1)]0,0) circle (0.4cm);
\end{scope}

\draw[->] (1.5,1.25)-- node[above]{$\phi_{ik}$} (3,1.25);

\end{tikzpicture}
\caption{The map $\phi_{ik}$ when $\bar {U_i}\cap \bar{U_k}= \emptyset$.}\label{figure:construction_phi}
\end{figure}
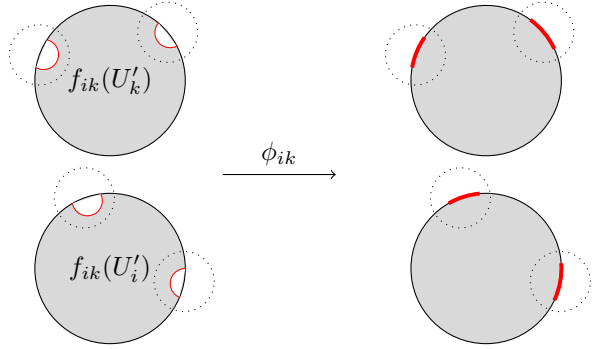

Next, suppose that $i\in I$ and $k\in I(i)$. Then $f_{ik}$ maps $U_i$ and $U_k$ to disks $D_{ik}$ and $D_{ki}$, respectively, that are tangent to each other. Applying the same procedure as above for $j\in I\setminus \{i,k\}$, we can find a map $\phi_{ik}$ that is $2K^4$-quasiconformal and maps $f_{ik}(U_i')$ onto $D_{ik}'$ and $f_{ik}(U_k')$ onto $D_{ki}'$; recall that $D_{ik}'$ is the component of $D_{ik}\setminus \gamma_{ik}$ that contains the center of the disk $D_{ik}$. See Figure \ref{figure:construction_psi} for an illustration.

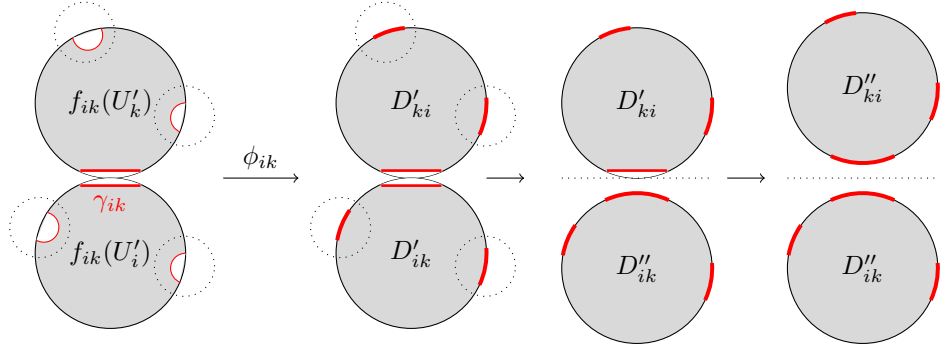
\begin{figure}
\begin{tikzpicture}
\begin{scope}[shift={(0,2)}]
\draw[fill=black!15]  (0,0) node {$f_{ik}(U_k')$} circle (1cm);
\fill[white] ([shift=(120:1)]0,0) arc (-170:26:0.2cm);
\draw[red] ([shift=(120:1)]0,0) arc (-170:30:0.2cm);
\draw[dotted] ([shift=(110:1)]0,0) circle (0.4cm);
\fill[white] ([shift=(0:1)]0,0) arc (90:246:0.2cm);
\draw[red] ([shift=(0:1)]0,0) arc (90:250:0.2cm);
\draw[dotted] ([shift=(-10:1)]0,0) circle (0.4cm);
\fill[white] ([shift=(-65:1)]0,0) arc (-65:-115:1);
\draw[red, line width =1pt] (-0.4,-0.9)-- (0.4,-0.9);
\end{scope}

\begin{scope}[shift={(0,0)}]
\draw[fill=black!15]  (0,0)node {$f_{ik}(U_i')$}  circle (1cm);
\fill[white,rotate=50] ([shift=(120:1)]0,0) arc (-170:26:0.2cm);
\draw[red,rotate=50] ([shift=(120:1)]0,0) arc (-170:30:0.2cm);
\draw[dotted,rotate=50] ([shift=(110:1)]0,0) circle (0.4cm);
\fill[white]([shift=(0:1)]0,0) arc (90:246:0.2cm);
\draw[red] ([shift=(0:1)]0,0) arc (90:250:0.2cm);
\draw[dotted] ([shift=(-10:1)]0,0) circle (0.4cm);
\fill[white] ([shift=(65:1)]0,0) arc (65:115:1);
\draw[red, line width =1pt] (-0.4,0.9)-- node[below]{$\gamma_{ik}$}(0.4,0.9);
\end{scope}

\begin{scope}[shift={(4,2)}]
\draw[fill=black!15] (0,0) node {$D_{ki}'$} circle (1cm);
\draw[red, line width=1.5] ([shift=(95:1)]0,0) arc (95:120:1cm);
\draw[dotted] ([shift=(110:1)]0,0) circle (0.4cm);
\draw[red, line width=1.5] ([shift=(4:1)]0,0) arc (4:-25:1cm);
\draw[dotted] ([shift=(-10:1)]0,0) circle (0.4cm);
\fill[white] ([shift=(-65:1)]0,0) arc (-65:-115:1);
\draw[red, line width =1pt] (-0.4,-0.9)-- (0.4,-0.9);
\end{scope}

\begin{scope}[shift={(4,0)}]
\draw[fill=black!15] (0,0) node {$D_{ik}'$} circle (1cm);
\draw[red, line width=1.5,rotate=50] ([shift=(95:1)]0,0) arc (95:120:1cm);
\draw[dotted,rotate=50] ([shift=(110:1)]0,0) circle (0.4cm);
\draw[red, line width=1.5] ([shift=(4:1)]0,0) arc (4:-25:1cm);
\draw[dotted] ([shift=(-10:1)]0,0) circle (0.4cm);
\fill[white] ([shift=(65:1)]0,0) arc (65:115:1);
\draw[red, line width =1pt] (-0.4,0.9)--(0.4,0.9);
\end{scope}

\draw[->] (1.5,1)-- node[above]{$\phi_{ik}$} (2.5,1);

\begin{scope}[shift={(7,2)}]
\draw[fill=black!15] (0,0) node {$D_{ki}'$} circle (1cm);
\draw[red, line width=1.5] ([shift=(95:1)]0,0) arc (95:120:1cm);
\draw[red, line width=1.5] ([shift=(4:1)]0,0) arc (4:-25:1cm);
\fill[white] ([shift=(-65:1)]0,0) arc (-65:-115:1);
\draw[red, line width =1pt] (-0.4,-0.9)-- (0.4,-0.9);
\end{scope}

\begin{scope}[shift={(7,-0.2)}]
\draw[fill=black!15] (0,0) node {$D_{ik}''$} circle (1cm);
\draw[red, line width=1.5,rotate=50] ([shift=(95:1)]0,0) arc (95:120:1cm);
\draw[red, line width=1.5] ([shift=(4:1)]0,0) arc (4:-25:1cm);
\draw[red, line width =1.5pt] ([shift=(65:1)]0,0) arc (65:115:1);
\end{scope}

\draw[dotted] (6,1)--(8,1);

\draw[->] (5,1)--(5.5,1);

\begin{scope}[shift={(10,2.2)}]
\draw[fill=black!15] (0,0) node {$D_{ki}''$}  circle (1cm);
\draw[red, line width=1.5] ([shift=(95:1)]0,0) arc (95:120:1cm);
\draw[red, line width=1.5] ([shift=(4:1)]0,0) arc (4:-25:1cm);
\draw[red, line width =1.5pt] ([shift=(-65:1)]0,0) arc (-65:-115:1);
\end{scope}

\begin{scope}[shift={(10,-0.2)}]
\draw[fill=black!15] (0,0) node {$D_{ik}''$} circle (1cm);
\draw[red, line width=1.5,rotate=50] ([shift=(95:1)]0,0) arc (95:120:1cm);
\draw[red, line width=1.5] ([shift=(4:1)]0,0) arc (4:-25:1cm);
\draw[red, line width =1.5pt] ([shift=(65:1)]0,0) arc (65:115:1);
\end{scope}

\draw[dotted] (9,1)--(11,1);
\draw[->] (8.2,1)--(8.7,1);
\end{tikzpicture}
\caption{The map $\phi_{ik}$ when $\bar {U_i}\cap \bar{U_k}\neq \emptyset$.}\label{figure:construction_psi}
\end{figure}

Finally, by Lemma \ref{lemma:qcpull} there exists a $2$-quasiconformal map of $\C$ that maps $D_{ik}'$ onto a disk $D_{ik}''$ and is the identity map on a half-plane that is disjoint from $D_{ik}$ and contains $D_{ki}$; see Figure \ref{figure:construction_psi}. Similarly, there exists another $2$-quasiconformal map that maps $D_{ki}'$ onto a disk $D_{ki}''$ and is the identity map on a half-plane that is disjoint from $D_{ki}$ and contains $D_{ik}''$. Let $g_{ik}$ be the composition of these two quasiconformal maps, which is $2$-quasiconformal in $\C$ and extends to a $2$-quasiconformal map of $\widehat \C$. The map $\psi_{ik}=g_{ik}\circ \phi_{ik}\circ f_{ik}$ is $4K^5$-quasiconformal and maps $U_i'$ and $U_k'$ to disks. This completes the proof.
\end{proof}

\section{Geometry of annuli and disks}

\subsection{Annulus width}
Let $(X,d)$ be a metric space. A \textit{(closed) annulus} in $X$ is a set $A\subset X$ of the form
$$A=A(x;r,R)= \{y\in X: r\leq d(x,y)\leq R\},$$
where $x\in X$, $0<r< R$, and $X\setminus \bar B_d(x,R)\neq \emptyset$. We define the \textit{width} of the annulus $A$ to be
$$w_A=\log(R/r).$$
Moreover for a set $K\subset X$, the \textit{width of $K$ relative to $A$} is defined as
$$w_{A}(K)=\inf\{ \log(R'/r'): A\cap K\subset A(x;r',R')\subset A\}.$$
By definition, if $A\cap K=\emptyset$, then $w_A(K)=0$. Also, if $w_A(K)>0$, then there exists an annulus $A'=A(x;r',R')$, concentric with $A$, such that $A\cap K\subset {A'}\subset  A$ and $w_A(K)=w_{A'}$. Namely, 
\begin{align*}
r'=r_A(K)\coloneqq \inf_{y\in A\cap K} d(x,y)\quad \text{and} \quad R'=R_A(K)\coloneqq \sup_{y\in A\cap K} d(x,y)
\end{align*}
so we have $A'=A(x; r_A(K), R_A(K))$ and 
$$w_A(K) = \log\frac{R_A(K)}{r_A(K)}.$$

For $x\in \C$ and $0<r<R<\infty$ we consider the Euclidean annulus 
$$A=A^e(x;r,R)= \{ y\in \C: r\leq |x-y|\leq R\}.$$
As a convention, in this case we will use the superscript $e$ in the above notions and we write $w^e_A$, $w^e_A(K)$, $r^e_A(K)$, $R^e_A(K)$. For $x\in \widehat \C$ and $0<r<R<\pi$ consider the spherical annulus
$$A=A(x;r,R)=\{y\in \widehat \C: r\leq \sigma(x,y)\leq R\}.$$
In the case of spherical annuli, we use the original notation $w_A$, $w_A(K)$, $r_A(K)$, $R_A(K)$ without any alteration. Observe that each Euclidean or spherical annulus has two boundary components that are circles.

The next elementary lemma is a modification of \cite{Bonk:uniformization}*{Lemma 8.6}. It essentially says that given an annulus $A$ and a collection $\mathcal K$ of continua in $\widehat \C$, one can find a large subannulus $A'$ of $A$ such that either all sets of $\mathcal K$, with the exception of at most one, have small width relative to $A'$ or (at least) two sets of $\mathcal K$ intersect both boundary components of $A'$. 

\begin{lemma}\label{lemma:subannulus}
Let $\mathcal K=\{K_i\}_{i\in I}$ be a collection of continua in $\widehat \C$. For every annulus $A$ in $\widehat \C$ there exists an annulus $A'\subset A$ that is concentric with $A$ such that 
$$w_{A'}\geq \min\{w_A,w_A^{1/9}\}$$
and one of the following two alternatives holds.
\begin{enumerate}[label=\normalfont(\roman*)]
\item\label{lemma:subannulus:1}  $w_{A'}(K_i)\leq w_{A'}^{1/3}$ for all $i\in I$ with the exception of at most one index.
\item\label{lemma:subannulus:2} There exist distinct $i_1,i_2\in I$ such that $K_{i_1}$ and $K_{i_2}$ intersect both boundary components of $A'$. 
\end{enumerate}
\end{lemma}

The exponent $1/3$ that appears in the first alternative at this point seems arbitrary, since the proof can be carried out with other restrictions as well. The importance of the number $1/3$ will become evident later in Proposition \ref{prop:modulus_upper}.

\begin{proof}
If $w_A(K_i)\leq w_A^{1/3}$ for all $i\in I$, then we set $A'=A$. Otherwise, there exists $i_1\in I$ such that $w_A\geq w_A(K_{i_1})>w_A^{1/3}$, so $w_A>w_{A}^{1/3}$ and $w_A>1$. Let $A_1$ be an annulus concentric with $A$ such that $A\cap K_{i_1}\subset A_1\subset A$ and $w_{A}(K_{i_1})=w_{A_1}$. Since $K_{i_1}$ is closed, it meets both boundary components of $A_1$. Moreover,  we have $w_{A_1}> w_A^{1/3}>w_A^{1/9}$. If $w_{A_1}(K_i)\leq w_{A_1}^{1/3}$ for all $i\in I\setminus\{i_1\}$, then we set $A'=A_1$. This completes the proof of the first alternative.

Otherwise, there exists $i_2\in I\setminus \{i_1\}$ such that $w_{A_1}(K_{i_2}) > w_{A_1}^{1/3}$. Let $A_2$ be an annulus concentric with $A$ such that $A_1\cap K_{i_2}\subset A_2\subset A_1$ and $w_{A_1}(K_{i_2})=w_{A_2}$. Since $K_{i_2}$ is closed, it meets both boundary components of $A_2$. Since $K_{i_1}$ is connected, it has the same property. We have $w_{A_2}>w_{A_1}^{1/3}>w_A^{1/9}$. We set $A'=A_2$ and this completes the proof.  
\end{proof}

\subsection{Width of disks relative to an annulus}
The statements that we prove in this section essentially imply that given a collection of disjoint disks in the sphere, at most two of them can have large width relative to a fixed annulus. 

\begin{proposition}\label{prop:annulus_euclidean_general}
Let $\alpha\geq 1$ and $A=A^e(x;r,R)$ be an annulus in $\C$. Let $K_1,K_2\subset \widehat \C$ be disjoint closed disks with $\diam_e(K_i)\geq \alpha^{-1}R$ and $\dist_e(K_i,x)\leq \alpha r$ for $i=1,2$. Then for every closed disk $D\subset \widehat \C\setminus (K_1\cup K_2)$ we have
$$w_A^e(D)\leq C(\alpha).$$ 
\end{proposition}

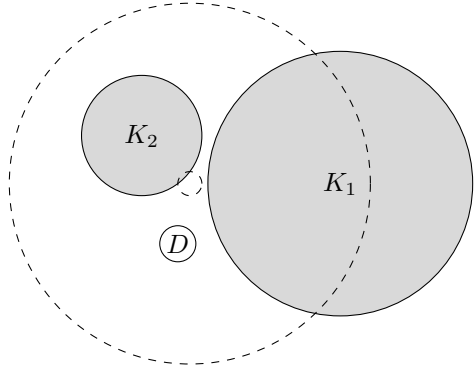
\begin{figure}
\begin{tikzpicture}[scale=0.8]

\draw[fill=black!15] (2.5,0) node{$K_1$}  circle (2.2cm);
\draw[fill=black!15] (-0.8,0.8) node {$K_2$} circle (1cm);

\draw (-0.2,-1) node {$D$} circle (0.3cm);

\draw[dashed] (0,0) circle (3cm);
\draw[dashed] (0,0) circle (0.2cm);

\end{tikzpicture}

\caption{The relative position of the annulus $A$ (dashed) and of the disks $K_1,K_2,D$ in Proposition \ref{prop:annulus_euclidean_general}.}\label{figure:annulus_euclidean_general}
\end{figure}

\begin{proof}
Note that $w_A^e(D)$ is invariant under translation and scaling. Thus, it suffices to show the statement for $x=0$ and $R=1$. Also observe that any closed disk $D\subset \widehat \C$ with $w_A^e(D)>0$ contains a disk $D'\subset D\cap  A$ with $w_A^e(D)=w_A^e(D')$. To see this, by rotating we may assume that $D$ is symmetric with respect to the real axis and the segment $[r_A^e(D),R_A^e(D)]$ is contained in $D$; then $D'$ is the disk centered at the midpoint of this segment with radius $\frac{1}{2}(R_A^e(D)-r_A^e(D))$. Thus, we may restrict to the case that $D\subset  A\setminus (K_1\cup K_2)$. See Figure \ref{figure:annulus_euclidean_general} for an illustration.

Suppose, for the sake of contradiction, that the statement fails for some $\alpha\geq 1$. Then for each $n\in \N$ there exists an annulus $A_n=A^e(0;r_n,1)$ and disjoint closed disks $K_1(n),K_2(n)\subset \widehat \C$ with $\diam_e(K_i(n)) \geq \alpha^{-1}$ and $\dist_e(K_i(n),0)\leq \alpha r_n$ for $i=1,2$ such that a disk $D_n\subset {A_n} \setminus (K_1(n)\cup K_2(n))$ satisfies $w^e_{A_n}(D_n)\to \infty$ as $n\to\infty$. For each $n\in \N$ we have
\begin{align}\label{prop:annulus_euclidean:contradiction}
\exp(w_{A_n}^e)=\frac{1}{r_n}\geq \frac{R_{A_n}^e(D_n)}{r_n}\geq \frac{R_{A_n}^e(D_n)}{r_{A_n}^e(D_n)}= \exp(w^e_{A_n}(D_n)).
\end{align}

Consider the scaling $\lambda_n(z)=z/R_{A_n}^e(D_n)$, $n\in \N$. Observe that, after passing to a subsequence, the disks $\lambda_n(D_n)$, $n\in \N$, converge to a disk $D_\infty$ (in the sense of convergence of centers and radii) with diameter equal to $1$ and  $0\in \partial D_\infty$. The latter is true because 
$$\dist_e(\lambda_n(D_n),0)= \frac{\dist_e(D_n,0)}{R_{A_n}^e(D_n)}= \frac{r_{A_n}^e(D_n)}{R_{A_n}^e(D_n)}=\exp(-w_{A_n}^e(D_n)),$$
which converges to $0$ as $n\to\infty$. On the other hand, for $i=1,2$ we have
$$\diam_e(\lambda_n(K_i(n)))= \frac{\diam_e(K_i(n))}{R_{A_n}^e(D_n)}\geq \alpha^{-1}$$
and 
$$\dist_e(\lambda_n(K_i(n)),0)= \frac{\dist_e(K_i(n),0)}{R_{A_n}^e(D_n)}\leq \frac{\alpha r_n}{R_{A_n}^e(D_n)},$$
which converges to $0$ as $n\to\infty$, by \eqref{prop:annulus_euclidean:contradiction}. Therefore, after passing to a further subsequence, we see that the sequences of disks $\lambda_n(K_i(n))$, $n\in \N$, for $i=1,2$, converge to closed disks $K_1(\infty),K_2(\infty)$ (in $\widehat \C$), respectively, with disjoint interiors and with $0\in \partial K_1(\infty)\cap \partial K_2(\infty)$.  The disks $D_\infty, K_1(\infty),K_2(\infty)$ have disjoint interiors and contain the point $0$ in their boundary. This is a contradiction. 
\end{proof}

\begin{lemma}\label{lemma:annulus_euclidean}
There exists a universal constant $\alpha_0>0$ with the following property. For every annulus $A$ in $\C$ and each pair of disjoint closed disks $K,L$ in $\widehat \C$ that intersect both boundary components of $A$ the following statement is true. If $L'$ is the reflection of the disk $L$ across $\partial K$, then every closed disk $D\subset \widehat \C \setminus (L\cup L')$ satisfies $w_A^e(D)\leq \alpha_0$. 
\end{lemma}

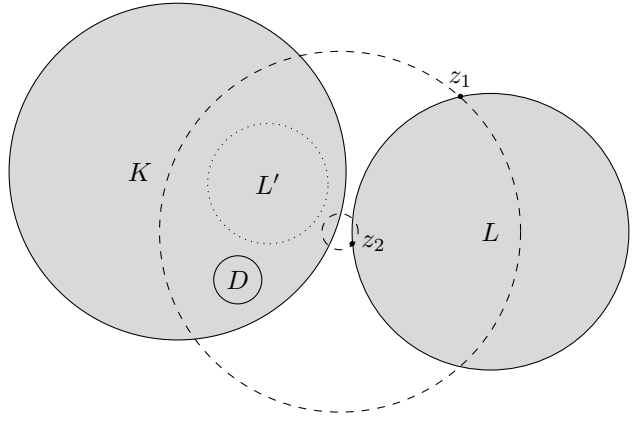
\begin{figure}
\begin{tikzpicture}[scale=0.8]

\draw[fill=black!15] (2.5,0) node{$L$}  circle (2.3cm);
\draw[fill=black!15] (-2.7,1) node[xshift=-0.5cm] {$K$} circle (2.8cm);
\draw[dotted] (-1.2,0.8) node {$L'$} circle (1cm);

\draw (-1.7,-0.8) node {$D$} circle (0.4cm);

\draw[dashed] (0,0) circle (3cm);
\draw[dashed](0,0) circle (0.3cm);

\draw[fill=black] (2,2.25) node[above] {$z_1$} circle (1pt);
\draw[fill=black] (0.2,-0.2) node[anchor=west] {$z_2$} circle (1pt);
\end{tikzpicture}

\caption{The relative position of the annulus $A$ (dashed) and of the disks $K,L,L',D$ in Lemma \ref{lemma:annulus_euclidean} and Lemma \ref{lemma:annulus_spherical}.}\label{figure:lemma:annulus_euclidean}
\end{figure}

\begin{proof}
Since $w_A^e(D)$ is invariant under translation and scaling, and circle reflections are natural under conjugation by M\"obius transformations (see \cite{Beardon:groups}*{Theorem 3.2.4}), we may assume that $A=A^e(0;r,1)$ and that $K,L$ are disjoint closed disks in $\widehat \C$ that intersect both boundary components of $A$. If $r>1/6$, then $w_A^e(D)\leq w_A^e\leq \log 6$ and we have nothing to show, so assume that $r\leq 1/6$. Then $\diam_e(L)\geq 1-r\geq 5/6$ and $\dist_e(L,0)\leq r$. Let $L'$ be the reflection of $L$ across $\partial K$. We show that $\diam_e(L')\geq d_0$ and $\dist_e(L',0)\leq c_0 r$ for some uniform constants $d_0,c_0>0$. By Proposition \ref{prop:annulus_euclidean_general}, this is sufficient for the desired conclusion. We consider two cases.

Suppose that $\partial K$ is a line containing the point $\infty$. Then $\diam_e(L')=\diam_e(L)\geq 5/6$ and $\dist_e(L',0)\leq 2\dist_e(K,0)+\dist_e(L,0)\leq 3r$. 

Next, suppose that $\partial K$ is a circle in $\C$ with center at a point $a\in \C$ and radius $\rho>0$. See Figure \ref{figure:lemma:annulus_euclidean} for an illustration. Note that the disk $K$ could be bounded or unbounded. Since the circle $|z|=r$ intersects both disks $K,L$, which are disjoint, there exists a point $z\in \partial K$ with $|z|=r$. This implies that $||a|-\rho|\leq r$. Similarly, since the circle $|z|=1$ intersects both $K$ and $L$, the circle $\partial K$ must also intersect the circle $|z|=1$. Thus, we have $\rho\geq (1-r)/2$ and 
\begin{align}\label{lemma:annulus_euclidean:ar}
|a|\geq \rho-r\geq (1-3r)/2 \geq 1/4\geq  3r/2> r.
\end{align}
As a consequence,
\begin{align}\label{lemma:annulus_euclidean:arho}
||a|^2-\rho^2| \leq r(|a|+\rho)\leq r(2|a|+r) \leq 3r|a|. 
\end{align}
Consider the reflection across $\partial K$, which is given by the formula
$$\phi(z)=a + \frac{\rho^2}{\bar z-\bar {a}}$$
for $z\in \widehat \C$. Let $z_1\in \partial L$ such that $|z_1|=1$. By \eqref{lemma:annulus_euclidean:arho} and \eqref{lemma:annulus_euclidean:ar}, we have
\begin{align*}
|\phi(z_1)| &= \frac{|a\bar z_1-|a|^2+\rho^2|}{|z_1-a|} \geq \frac{ |a| - 3r|a| }{1+|a|}\geq  \frac{|a|(1-3r)}{5|a|}\geq \frac{1}{10}.
\end{align*}
Next, let $z_2\in \partial L$ with $|z_2|=r$. With similar estimates as above we have
\begin{align*}
|\phi(z_2)| \leq \frac{|a|r + 3r|a|}{|a|-r}\leq \frac{4r|a|}{|a|/3}= 12 r.
\end{align*}
We suppose that $12r\leq 1/20$, otherwise $w_A^e\leq \log 240$ and there is nothing to show. We have $\dist_e(L',0)\leq |\phi(z_2)|\leq 12 r$ and $\diam_e(L')\geq |\phi(z_1)|-|\phi(z_2)|\geq 1/10-12r \geq 1/20$. This completes the proof.
\end{proof}

We now obtain a version of Lemma \ref{lemma:annulus_euclidean} for spherical annuli. 

\begin{lemma}\label{lemma:annulus_spherical}
There exists a universal constant $\beta_0>0$ with the following property. For every annulus $A$ in $\widehat\C$ and each pair of disjoint closed disks $K,L$ in $\widehat \C$ that intersect both boundary components of $A$ the following statement is true. If $L'$ is the reflection of the disk $L$ across $\partial K$, then every closed disk $D\subset \widehat \C\setminus (L\cup L')$ satisfies $w_A(D)\leq \beta_0$. 
\end{lemma}

\begin{proof}
Using an isometry of $\widehat \C$, we may assume that the annulus $A$ is centered at $0$. Let $A=A(0;r,R)$, where $0<r<R<\pi$, and $K,L,D$ be as in the statement. Suppose that $w_A(D)>\log 2$, otherwise there is nothing to show. Note that $w_A\geq w_A(D)>\log 2$, so $R>2r$.  Then $D$ must intersect the annulus $B=A(0;r,R/2)$. We have $R_A(D)\leq 2R_B(D)$ and $r_A(D)=r_B(D)$, so 
$$w_A(D)=\log \frac{R_A(D)}{r_A(D)} \leq \log 2+ w_B(D).$$
The disks $K,L$ intersect both boundary components of $B$. Note that $0<r<R/2<\pi/2$, so $B\subset \D$ and the Euclidean metric is comparable to the spherical metric in $B$. More specifically, we have
$$ |z-w|\leq \sigma(z,w)\leq 2|z-w|$$
for each $z,w\in \D$. By viewing $B$ as a Euclidean annulus, we have
\begin{align*}
R_B(D)=\sigma(0,R_B^e(D))\leq 2R_B^e(D)\quad \text{and}\quad r_B(D)=\sigma(0,r_B^e(D))\geq r_B^e(D).
\end{align*}
We conclude that $w_B(D)\leq \log 2+ w_B^e(D)$. Altogether, we have $w_A(D)\leq 2\log 2 +w_B^e(D)$. By Lemma \ref{lemma:annulus_euclidean}, $w_B^e(D)\leq \alpha_0$ and this completes the proof.
\end{proof}

\subsection{Schottky groups}\label{section:schottky}

Let $\mathcal K=\{K_i\}_{i\in I}$ be a collection of closed disks in $\widehat \C$ with disjoint interiors. The \textit{Schottky group} of $\mathcal K$ is the group of M\"obius and anti-M\"obius transformations generated by reflections in the disks of $\mathcal K$ and is denoted by $\Gamma(\mathcal K)$.  Let $\phi_i$ be the reflection in $K_i$, $i\in I$. Each element $\phi\in \Gamma(\mathcal K)$ that is not the identity map can be expressed uniquely as $\phi=\phi_{i_1}\circ \dots \circ \phi_{i_n}$, $n\in \N$, where $i_j\neq i_{j+1}$ for $j\in \{1,\dots,n-1\}$. This is called the \textit{reduced form} of $\phi$. We define the \textit{length} of $\phi$ to be $l(\phi)=n$. The length of the identity map is defined to be $0$. See \cite{BonkKleinerMerenkov:schottky}*{Section 3} or \cite{NtalampekosYounsi:rigidity}*{Section 7} for further details.

The next proposition strengthens the conclusion of Lemma \ref{lemma:annulus_spherical} and asserts that if two disks $K_1,K_2$ are large relative to a fixed annulus $A$, then not only each disk $D\subset \widehat \C\setminus (K_1\cup K_2)$ has small width relative to $A$, but actually the entire orbit of $D$ under the Schottky group of $\{K_1,K_2\}$ retains the same property.

\begin{proposition}\label{prop:reflect}
Let $A$ be an annulus in $\widehat \C$ and $K_1,K_2$ be a pair of disjoint closed disks in $\widehat \C$ that intersect both boundary components of $A$. Let $\Gamma$ be the Schottky group of $\{K_1,K_2\}$. For each closed disk $D\subset \widehat \C\setminus (K_1\cup K_2)$ and for each $\phi\in \Gamma$ we have
$$w_A(\phi(D))\leq \beta_0,$$
where $\beta_0$ is the universal constant appearing in Lemma \ref{lemma:annulus_spherical}
\end{proposition}

\begin{proof}
Let $\Omega=\widehat \C\setminus (K_1\cup K_2)$. Let $\phi_i$ be the reflection across $\partial K_i$, $i=1,2$. By Lemma \ref{lemma:annulus_spherical} we have $w_A(D)\leq \beta_0$ for each closed disk $D\subset \widehat \C\setminus (K_2\cup \phi_1(K_2))$. In particular, since $\phi_1(K_2)\subset K_1$, this is true for each $D\subset \Omega$, so we have verified the statement when $\phi$ is the identity map. Next, for each $D\subset \Omega$ we have $\phi_1(D)\subset \phi_1(\Omega)= \inter(K_1) \setminus \phi_1(K_2)\subset \widehat \C\setminus (K_2\cup \phi_1(K_2))$. Thus, by the above, $w_A(\phi_1(D))\leq \beta_0$. With the same argument, one obtains $w_A(\phi_2(D))\leq \beta_0$ for each $D\subset \Omega$ and we have verified the statement when $l(\phi)=1$. 

Let $\phi\in \Gamma$ with $l(\phi)=n\geq 2$ be represented in reduced form as $\phi=\phi_{i_1}\circ \dots \circ \phi_{i_n}$, where $i_j\in \{1,2\}$ and $i_j\neq i_{j+1}$ for $j\in \{1,\dots,n-1\}$. Note that $\phi(\Omega)$ has two complementary disks, $\phi(K_1)$ and $\phi(K_2)$, and 
$$\phi(\Omega)\subset \phi_{i_1}\circ \dots \circ \phi_{i_{n-1}} ( K_{i_n})$$
because $\phi_{i_n}(\Omega)\subset K_{i_n}$. If $w_A(\phi_{i_1}\circ \dots \circ \phi_{i_{n-1}} ( K_{i_n})) =0$, then we also have $w_A(\phi(D))=0$ for all $D\subset \Omega$. In this case we define $A(\phi)=\emptyset$. Suppose that $w_A(\phi_{i_1}\circ \dots \circ \phi_{i_{n-1}} ( K_{i_n}))>0$. Then there exists an annulus $A(\phi)$ concentric with $A$ such that 
$$A\cap \phi_{i_1}\circ \dots \circ \phi_{i_{n-1}} ( K_{i_n})\subset A(\phi)\subset A\,\,\,\text{and}\,\,\, w_{A(\phi)}= w_A(\phi_{i_1}\circ \dots \circ \phi_{i_{n-1}} ( K_{i_n})).$$
The disks $\phi_{i_1}\circ \dots \circ \phi_{i_{n-1}} ( K_{1})$ and $\phi_{i_1}\circ \dots \circ \phi_{i_{n-1}} ( K_2)$ are disjoint. We show by induction on the length of $\phi$ that each of these two disks intersects both boundary components of $A(\phi)$. 

For $n=2$, let $\phi=\phi_{i_1}\circ \phi_{i_2}$ and consider the disks $\phi_{i_1}(K_1)$ and $\phi_{i_1}(K_2)$. Without loss of generality, $i_2=2$, so $i_1=1$ and we wish to show that $\phi_1(K_1)$ and $\phi_1(K_2)$ intersect both boundary components of $A(\phi)$. By the definition of $A(\phi)$ this is immediate for the disk $\phi_1(K_2)$. Note that $\phi_1(K_1)=\widehat \C\setminus \inter(K_1)\supset K_2$. By assumption, $K_2$ intersects each of the boundary components of the initial annulus $A$, so the same is true for the subannulus $A(\phi)$. Thus, $\phi_1(K_1)$ has the same property, as desired.

Now, suppose that the claim is true whenever $l(\phi)=n$ for some $n\geq 2$. Let $\phi=\phi_{i_1}\circ \dots\circ \phi_{i_n}\circ \phi_{i_{n+1}}$ such that $A(\phi)\neq \emptyset$. Without loss of generality, $i_{n+1}=1$. Let $\psi=\phi_{i_1}\circ \dots\circ \phi_{i_n}$. Then the disk $\psi(K_1)$ intersects both boundary components of $A(\phi)$ by definition. We have $\psi(K_2)=\phi_{i_1}\circ \dots \circ \phi_{i_{n-1}}( \widehat \C\setminus \inter(K_2))$. By the induction assumption, the disjoint disks $\phi_{i_1}\circ \dots \circ \phi_{i_{n-1}}(K_1)$ and $\phi_{i_1}\circ \dots \circ \phi_{i_{n-1}}(K_2)$ intersect both boundary components of $A(\psi)$. Thus, $\phi_{i_1}\circ \dots \circ \phi_{i_{n-1}}(K_2)$ does not contain any of the boundary components of $A(\psi)$ and $\phi_{i_1}\circ \dots \circ \phi_{i_{n-1}}(\partial K_2)$ intersects both boundary components of $A(\psi)$. Hence, $\psi(K_2)$ intersects both boundary components of $A(\psi)$. Observe that $A(\psi)\supset A(\phi)$, since $\phi_{i_1}\circ \dots \circ \phi_{i_{n-1}}(K_2)\supset \phi_{i_1}\circ \dots\circ \phi_{i_{n-1}}\circ  \phi_{i_n}(K_1)$.
Thus, $\psi(K_2)$ intersects both boundary components of $A(\phi)$, as desired. We have established the inductive claim.

Now, let $\phi\in \Gamma$ with $l(\phi)=n \geq 2$ and $\phi=\phi_{i_1}\circ \dots \circ \phi_{i_n}$. Also, let $\psi=\phi_{i_1}\circ \dots \circ \phi_{i_{n-1}}$. Without loss of generality, $i_n=2$. The disks $\psi ( K_{1})$ and $\psi (K_2)$ are disjoint and intersect both boundary components of $A(\phi)$, by the above claim. The reflection in $\psi ( K_2)$ is given by  $\psi\circ \phi_2 \circ \psi^{-1}$. By Lemma \ref{lemma:annulus_spherical}, each closed disk 
$$D\subset \widehat \C\setminus (\psi ( K_{1}) \cup  \psi\circ \phi_2(K_1))=\widehat \C\setminus (\psi ( K_{1}) \cup  \phi(K_1)) $$
satisfies $w_{A(\phi)}(D)\leq \beta_0$. If $D\subset \Omega$, then 
$$\phi(D)\subset \widehat \C\setminus (\phi(K_1)\cup \phi(K_2)) \subset \widehat \C\setminus (\phi(K_1)\cup \psi(K_1)),$$
given that $\psi(K_1) \subset \psi(\widehat \C\setminus \inter(K_2))=\phi(K_2)$. Therefore, $w_{A(\phi)}(\phi(D))\leq \beta_0$. Moreover, since $\phi(D)=\psi\circ \phi_2(D)\subset \psi(K_2)$, by the choice of the annulus $A(\phi)$, we have $w_{A(\phi)} (\phi(D))= w_A(\phi(D))$; recall that $A(\phi)$ is the smallest subannulus of $A$ that contains $A\cap \psi(K_2)$. This completes the proof.
\end{proof}

We include some further properties of Schottky groups in this section that will be used in the proof of the main theorem. Suppose that we are given a bijection from a collection of disks $\mathcal K=\{K_i\}_{i\in I}$ with disjoint interiors to another such collection $\mathcal K'=\{K_i'\}_{i\in I}$ so that $K_i$ corresponds to $K_i'$ for each $i\in I$. If $\phi=\phi_{i_1}\circ \dots \circ \phi_{i_n}\in \Gamma(\mathcal K)$ is expressed in reduced form, then we define $\phi'=\phi_{i_1}'\circ \dots \circ \phi_{i_n}'\in \Gamma(\mathcal K')$.

\begin{lemma}\label{lemma:schottky:limit}
Let $\mathcal K=\{K_i\}_{i\in I}$ be a finite collection of disjoint closed disks in $\widehat \C$ and let $S=\widehat \C\setminus \bigcup_{i\in I}\inter(K_i)$. 
\begin{enumerate}[label=\normalfont(\roman*)]
\item\label{lemma:schottky:limit:1} The set $S_\infty= \bigcup_{\phi\in \Gamma(\mathcal K)} \phi(S)$ is a domain and $\widehat \C\setminus S_\infty$ is a totally disconnected compact set. If $\# I=2$, then $\widehat \C\setminus S_\infty$ consists of two points.
\end{enumerate}
Let $\mathcal K'=\{K_i'\}_{i\in I}$ be another collection of disjoint closed disks in $\widehat \C$ and let $f\colon \widehat \C\to \widehat \C$ be a homeomorphism such that $K_i'=f(K_i)$ for each $i\in I$. 
\begin{enumerate}[label=\normalfont(\roman*)]\setcounter{enumi}{1}
\item\label{lemma:schottky:limit:2}  There exists an extension of $f|_S$ to a homeomorphism $F \colon \widehat \C \to \widehat \C$ such that $F\circ \phi = \phi'\circ F$ for each $\phi\in \Gamma(\mathcal K)$.
\end{enumerate}
Let $\Omega\subset \widehat \C$ be a finitely connected domain such that the components of $\widehat \C\setminus \Omega$ are the disks $K_i$, $i\in I$, together with a finite collection of disjoint compact sets $\{L_j\}_{j\in J}$. 
\begin{enumerate}[label=\normalfont(\roman*)]\setcounter{enumi}{2}
\item\label{lemma:schottky:limit:3}  The set  
$$G=S_\infty \setminus \bigcup\{\phi(L_j): j\in J,\,\, \phi\in \Gamma(\mathcal K)\}$$
is a domain.
\item\label{lemma:schottky:limit:4} If $f$ is $H$-quasiconformal in $\Omega$ for some $H\geq 1$, then the extension $F$ is $H$-quasiconformal in $G$.
\end{enumerate}
\end{lemma}

\begin{figure}
	\includegraphics[scale=0.3]{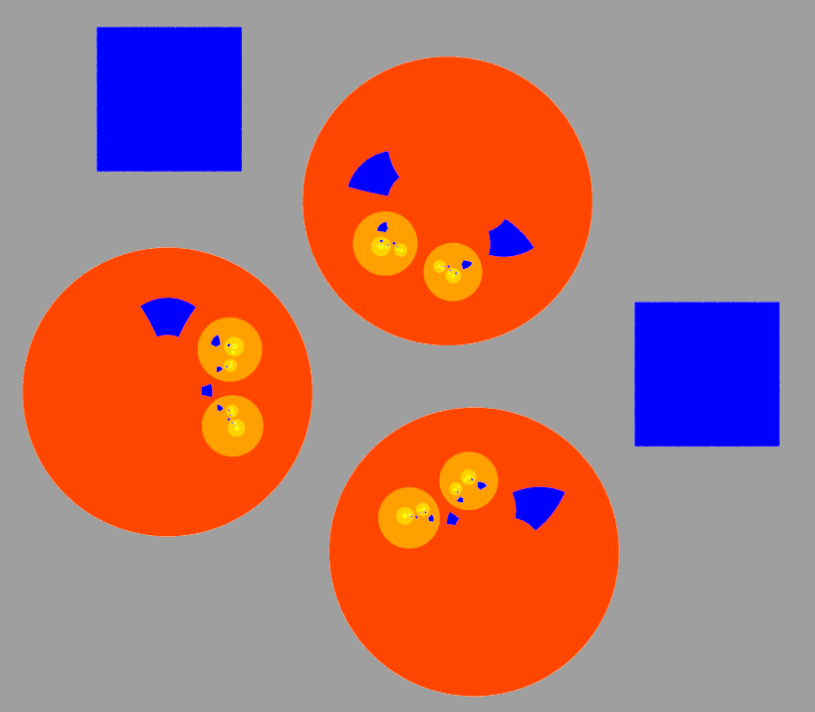}
	\caption{A finitely connected domain $\Omega$ whose complementary components are the three largest disks and the two squares. The set $\widehat \C\setminus S_\infty$ appearing in Lemma \ref{lemma:schottky:limit} \ref{lemma:schottky:limit:1} is a Cantor set. The set $\widehat \C\setminus G$ in Lemma \ref{lemma:schottky:limit} \ref{lemma:schottky:limit:3} is the union of this Cantor set with the squares and their reflections.}\label{figure:schottky:limit}
\end{figure}

The various sets appearing in the lemma are illustrated in Figure \ref{figure:schottky:limit}.

\begin{proof}
Let $\phi_i$ be the reflection in $K_i$, $i\in I$. It is shown in \cite{NtalampekosYounsi:rigidity}*{Corollary 7.4} that the set $\widehat \C\setminus S_\infty$ is equal to the intersection over $n\in \N$ of the disjointed unions
\begin{align*}
\bigsqcup_{i_j\neq i_{j+1}} \phi_{i_1}\circ \dots \circ \phi_{i_n}(K_{i_{n+1}}).
\end{align*}
Given that $I$ is finite, each such union is compact. Therefore, $\widehat \C\setminus S_\infty$ is compact. Also, again from \cite{NtalampekosYounsi:rigidity}*{Corollary 7.4}, for each point $z\in \widehat \C\setminus S_\infty$ there exists a unique sequence $\{i_j\}_{j\in \N}$ such that $i_j\in I$ and $i_j\neq i_{j+1}$ for each $j\in \N$, and we have
$$\{z\}=\bigcap_{j=1}^\infty \phi_{i_1}\circ \dots \circ \phi_{i_j}(K_{i_{j+1}}).$$
This implies that $\widehat \C\setminus S_\infty$ is totally disconnected. Hence, its complement is a domain; see \cite{Wilder:topology}*{Lemma II.5.20}. If $\#I=2$, then there are only two sequences $\{i_j\}_{j\in \N}$ as above, so $\widehat \C\setminus S_\infty$ has only two elements. This proves \ref{lemma:schottky:limit:1}. Part \ref{lemma:schottky:limit:2} is proved in \cite{NtalampekosYounsi:rigidity}*{Lemma 7.5}; see also \cite{BonkKleinerMerenkov:schottky}*{Lemma 5.1}.

Now we establish part \ref{lemma:schottky:limit:3}. First observe that the open sets $\phi(\inter(S))$, $\phi\in \Gamma(\mathcal K)$, are disjoint. This implies that the compact sets $\phi(L_j)$, $j\in J$, $\phi\in \Gamma(\mathcal K)$, are disjoint. By the argument used in \ref{lemma:schottky:limit:1}, observe that for each $j\in J$ the set $\phi(L_j)$ converges as $l(\phi)\to \infty$ to the boundary of the open set $S_\infty$, i.e., to $\widehat \C\setminus S_\infty$. This implies that the set $\bigcup\{\phi(L_j): j\in J,\,\, \phi\in \Gamma(\mathcal K)\}$ is relatively closed in $S_\infty$. Thus, $G$ is open. Also, any two points of $G$ can be connected by a path in the domain $S_\infty$. This path intersects at most finitely many sets $\phi(L_j)$, where $j\in J$, $\phi\in \Gamma(\mathcal K)$, thus, it can be modified to a path that avoids each of these sets and stays in $G$. We conclude that $G$ is connected. This proves \ref{lemma:schottky:limit:3}.

Before proving \ref{lemma:schottky:limit:4}, we show that $G$ agrees with the set
\begin{align}\label{lemma:schottky:limit:g}
\widetilde G= \bigcup \{\phi(\Omega)\cup \phi(\partial K_i): i\in I,\,\, \phi\in \Gamma(\mathcal K)\}.
\end{align}
To prove this, observe that
$$\Omega\cup \bigcup_{i\in I}\partial K_i= S\setminus \bigcup_{j\in J}L_j= S \setminus \bigcup \{\psi(L_j): j\in J,\,\, \psi\in \Gamma(\mathcal K)\},$$
where the last equality follows from the fact that $\psi(L_j)\cap S=\emptyset$ if $j\in J$ and $\psi\in \Gamma(\mathcal K)$ is not the identity map. Therefore, for $\phi\in \Gamma(\mathcal K)$ we have
\begin{align*}
\phi \left( \Omega\cup \bigcup_{i\in I}\partial K_i\right)&= \phi(S) \setminus \bigcup \{\phi\circ \psi(L_j): j\in J,\,\, \psi\in \Gamma(\mathcal K)\}\\
&= \phi(S) \setminus \bigcup \{\psi(L_j): j\in J,\,\, \psi\in \Gamma(\mathcal K)\}.
\end{align*}
Taking the union over $\phi\in \Gamma(\mathcal K)$, we have
\begin{align*}
\widetilde G= \bigcup \{ \phi(\Omega\cup \bigcup_{i\in I}\partial K_i): \phi\in \Gamma(\mathcal K)\}= S_\infty \setminus \bigcup \{\psi(L_j): j\in J,\,\, \psi\in \Gamma(\mathcal K)\}=G.  
\end{align*}

Finally, we show the quasiconformality of $F$ on $G$. For each $\phi\in \Gamma(\mathcal K)$, by \ref{lemma:schottky:limit:2}, for $z\in \phi(\Omega)$ we have
$$F(z)= \phi'\circ F\circ \phi^{-1}(z).$$
Since $\phi$ and $\phi'$ are anti-M\"obius transformations and $F$ is $H$-quasiconformal in $\Omega$, we conclude that $F$ is $H$-quasiconformal in $\phi(\Omega)$ for each $\phi\in \Gamma(\mathcal K)$. 

For each $i\in I$ and $\phi\in \Gamma(\mathcal K)$ the circle $\phi(\partial K_i)$ is the common boundary component of the domains $\phi(\Omega)$ and $\phi(\phi_i(\Omega))$. Since $\Omega$ is finitely connected, $\phi(\partial K_i)$ has a positive distance from other boundary components of $\phi(\Omega)$ and $\phi(\phi_i(\Omega))$. Thus, in view of \eqref{lemma:schottky:limit:g}, $U=\phi(\Omega)\cup \phi(\phi_i(\Omega))\cup \phi (\partial K_i)$ is an open subset of $G$ such that $F$ is $H$-quasiconformal in $U\setminus \phi(\partial K_i)$. By the quasiconformal removability of circles (see \cite{LehtoVirtanen:quasiconformal}*{Theorem I.8.3}), $F$ is $H$-quasiconformal in $U$. Since $i\in I$ and $\phi\in \Gamma(\mathcal K)$ were arbitrary, $F$ is $H$-quasiconformal in $G$. This completes the proof.
\end{proof}

\section{Transboundary modulus estimates}

\subsection{Classical and transboundary modulus}\label{section:transboundary}

Let $\Gamma$ be a family of curves in $\widehat \C$. A Borel function $\rho\colon \widehat \C\to [0,\infty]$ is \textit{admissible} for $\Gamma$ if $\int_\gamma \rho\, ds\geq 1$ for every locally rectifiable path $\gamma\in \Gamma$. The \textit{(conformal) modulus} of $\Gamma$ is defined as
$$\Mod \Gamma=\inf_\rho \int\rho^2\, d\Sigma,$$
where the infimum is taken over all admissible functions $\rho$ for $\Gamma$. If there exists no admissible function for $\Gamma$, then by convention $\Mod \Gamma=\infty$. 

Let $\Omega\subset \widehat \C$ be a domain and $\mathcal K=\{K_i\}_{i\in I}$ be a collection of disjoint compact subsets of $\Omega$ such that the set $K=\bigcup_{i\in I}K_i$ is relatively closed in $\Omega$.   Let $\gamma\colon [a,b]\to \widehat \C$ be a curve. The \textit{trace} of $\gamma$ is the set $\gamma([a,b])$ and is denoted by $|\gamma|$. The set $\gamma^{-1}(\Omega\setminus K)$ is relatively open in $[a,b]$, so it is equal to the union of countably many intervals $A_j$, $j\in J$, that are relatively open in $[a,b]$. We say that $\gamma$ is \textit{locally rectifiable} in $\Omega\setminus K$ if $\gamma|_{A_j}$ is locally rectifiable for each $j\in J$. In this case, if $\rho\colon \Omega\setminus K\to[0,\infty]$ is a Borel function, we define
$$\int_{\gamma} \rho\chi_{\Omega\setminus K}\, ds= \sum_{j\in J} \int_{\gamma|_{A_j}}\rho\, ds.$$

A \textit{transboundary mass distribution} $\rho$ in $\Omega$ consists of a Borel function $\rho\colon \Omega\setminus K\to [0,\infty]$ and a collection of weights $\rho(K_i)\geq 0$, $i\in I$. The \textit{mass} of $\rho$ is defined as
$$\int_{\Omega\setminus K} \rho^2\, d\Sigma+ \sum_{i\in I}\rho(K_i)^2.$$
The transboundary mass distribution $\rho$ is \textit{admissible} for a curve family $\Gamma$ in $\widehat \C$ if 
$$\int_{\gamma} \rho\chi_{\Omega\setminus K}\, ds+ \sum_{i: |\gamma|\cap K_i\neq \emptyset} \rho(K_i)\geq 1$$
for every $\gamma\in \Gamma$ that is locally rectifiable in $\Omega\setminus K$. The \textit{transboundary modulus} of $\Gamma$ with respect to $\Omega$ and $\mathcal K$ is defined as
$$\Mod_{\Omega,\mathcal K} \Gamma= \inf_{\rho} \left\{ \int_{\Omega\setminus K} \rho^2\, d\Sigma + \sum_{i\in I}\rho(K_i)^2\right\},$$
where the infimum is taken over all transboundary mass distributions in $\Omega$ that are admissible for $\Gamma$. If there is no admissible transboundary mass distribution, we define $\Mod_{\Omega,\mathcal K} \Gamma=\infty$. 

Transboundary modulus is conformally invariant, as was observed by Schramm \cite{Schramm:transboundary}*{Lemma 1.1}. The next statement is a generalization of \cite{Bonk:uniformization}*{Lemma 6.3} and can be found in \cite{Rehmert:thesis}*{Lemma 3.1.3}.

\begin{lemma}[Invariance of transboundary modulus]\label{lemma:invariance}
Let $\Omega,\Omega'\subset \widehat \C$ be domains and let $\mathcal K=\{K_i\}_{i\in I}$ be a countable collection of disjoint compact subsets of $\Omega$ such that $\bigcup_{i\in I}K_i$ is relatively closed in $\Omega$. Suppose that $f\colon \Omega\to \Omega'$ is a homeomorphism that is $H$-quasiconformal on $\Omega\setminus \bigcup_{i\in I}K_i$ for some $H\geq 1$ and let $\mathcal K'=\{f(K_i)\}_{i\in I}$. Then for every family of curves $\Gamma$ in $\Omega$ and for  $f(\Gamma)=\{f\circ \gamma: \gamma\in \Gamma\}$ we have
$$H^{-1}\cdot\Mod_{\Omega,\mathcal K}\Gamma\leq  \Mod_{\Omega',\mathcal K'}f(\Gamma)\leq H \cdot \Mod_{\Omega,\mathcal K}\Gamma.$$
\end{lemma}

\subsection{Fat sets}
Let $\tau>0$. A measurable set $K\subset \widehat \C$ is \textit{$\tau$-fat} if for each $x\in K$ and for each ball $B(x,r)$ that does not contain $K$ we have $\Sigma(B(x,r)\cap K)\geq \tau r^2$. A set is \textit{fat} if it is $\tau$-fat for some $\tau>0$. Note that points are automatically fat for each $\tau>0$. A more modern terminology for fatness is Ahlfors $2$-regularity, but we prefer to stick to the original terminology that was used by Schramm \cite{Schramm:transboundary} and Bonk \cite{Bonk:uniformization}. In fact, Schramm defines fatness using the Euclidean metric, but a set is fat in the Euclidean metric if and only if it is fat in the spherical metric, quantitatively; see the proof of \cite{Ntalampekos:uniformization_packing}*{Lemma 2.3}.

\begin{lemma}\label{lemma:fat_count}
Let $\tau>0$ and $\mathcal K=\{K_i\}_{i\in I}$ be a collection of disjoint $\tau$-fat continua in $\widehat \C$. Then for each set $A\subset \widehat \C$ and $t>0$ the set
$$\{i\in I: K_i\cap A\neq \emptyset \,\,\, \text{and}\,\,\, \diam (K_i)\geq t\diam(A)\}$$
has at most $C(\tau,t)$ elements. 
\end{lemma}

See \cite{MaioNtalampekos:packings}*{Lemma 2.6} for the proof of the statement with the Euclidean metric. The proof remains unchanged if one uses the spherical metric. See also \cite{HakobyanLi:qs_embeddings}*{Lemma 4.8} for another argument. A proof of the next elementary lemma can be found in \cite{Ntalampekos:CarpetsThesis}*{Remark 2.3.5} or \cite{Bonk:uniformization}*{(42), p.~615}.

\begin{lemma}\label{lemma:fat_annulus}
Let $\tau>0$ and $K\subset \widehat \C$ be a $\tau$-fat continuum. Let $A(x;r,R)$ be an annulus in $\widehat \C$ such that $K$ intersects both $\partial B(x,r)$ and $\partial B(x,R)$. Then 
$$\Sigma ( A(x;r,R)\cap K)\geq c(\tau)(R-r)^2.$$
\end{lemma}

An elementary property of (closed) geometric disks in $\widehat \C$ is that they are $\pi^{-1}$-fat; see the argument in \cite{Bonk:uniformization}*{p.~609} and adapt it to the spherical metric. Schramm proved the quasiconformal invariance of fatness, which yields the next statement.

\begin{lemma}[\cite{Schramm:transboundary}*{Corollary 2.3}]\label{lemma:fat_quasidisk}
Each (closed) quasidisk is fat, quantitatively. 
\end{lemma}

Variants of the next lemma have appeared repeatedly in the literature; see e.g.\ \cite{Bojarski:inequality}*{Lemma 4.2}. The current statement appears in \cite{Ntalampekos:uniformization_packing}*{Lemma 2.7}.

\begin{lemma}\label{lemma:bojarski}
Let $I$ be a countable index set and $a\geq 1$. Suppose that $\{b_i\}_{i\in I}$ is a collection of non-negative numbers, $\{K_i\}_{i\in I}$ is a family of measurable sets, and $\{B_i\}_{i\in I}$ is a family of disks in $\widehat \C$ such that $K_i\subset B_i$ and $\Sigma(B_i)\leq a\cdot \Sigma(K_i)$ for each $i\in I$. Then 
$$\left\|\sum_{i\in I}b_i\chi_{B_i} \right\|_{L^2(\widehat \C)} \leq C(a)\left\|\sum_{i\in I}b_i\chi_{K_i} \right\|_{L^2(\widehat \C)}.$$
\end{lemma}

\subsection{Lower bound for transboundary modulus}\label{section:lower}
The results in this and in the next section provide sufficient conditions for transboundary modulus to behave similarly to conformal modulus. The following proposition appears in a weaker form in \cite{HakobyanLi:qs_embeddings}*{Lemma 4.6} and in \cite{Rehmert:thesis}*{Lemma 3.3.5}, and its proof is essentially a modification of the proof of \cite{Bonk:uniformization}*{Prop.\ 8.1}. 

\begin{proposition}[Comparison of classical and transboundary modulus]\label{prop:modulus_compare}
Let $\tau>0$. Let $\Omega\subset \widehat \C$ be a domain and $\mathcal K=\{K_i\}_{i\in I}$ be a collection of disjoint $\tau$-fat continua in $\Omega$ such that $K=\bigcup_{i\in I}K_i$ is relatively closed in $\Omega$. Then for each curve family $\Gamma$ in $\widehat \C$ we have
$$\min\{1, \Mod \Gamma \}\leq c(\tau) \Mod_{\Omega,\mathcal K} \Gamma.$$
\end{proposition}

\begin{proof}
Let $m>0$ be a constant to be determined, depending only on $\tau$. If $\Mod_{\Omega,\mathcal K}\Gamma\geq m$ the statement is true for $c(\tau)\geq m^{-1}$, so suppose that $\Mod_{\Omega,\mathcal K}\Gamma<m$. Also, if $\Gamma$ contains a constant curve that intersects a set $K_i$, $i\in I$, then any transboundary mass distribution $\rho$ that is admissible for $\Gamma$ must satisfy $\rho(K_i)\geq 1$, so $\Mod_{\Omega,\mathcal K}\Gamma \geq 1$. In this case, the desired statement is true for $c(\tau)\geq 1$. Hence we suppose that no constant curve (if any) in $\Gamma$  intersects $K_i$, $i\in I$. 

Let $\rho$ be a transboundary mass distribution that is admissible for $\Gamma$ with 
$$\int_{\Omega\setminus K} \rho^2\, d\Sigma +\sum_{i\in I} \rho(K_i)^2<m.$$
In particular, $\rho(K_i)\neq 0$ for at most countably many $i\in I$ and
\begin{align}\label{prop:modulus_compare:rho_m}
\rho(K_i)<m^{1/2} \,\,\, \text{for each $i\in I$}.
\end{align}
Let $J=\{ i\in I: \diam(K_i)>0\}$, which is a countable set, as can be seen directly by the fatness assumption or by Lemma \ref{lemma:fat_count}. Let $\Gamma_0$ be the family of curves in $\Gamma$ that intersect some of the countably many singletons $K_i$, $i\in I\setminus J$, that satisfy $\rho(K_i)\neq 0$. By our assumption, $\Gamma_0$ is a family of non-constant curves passing through a countable collection of points. Therefore, $\Mod \Gamma_0=0$; see \cite{Vaisala:quasiconformal}*{Section 7.9}. Moreover, 
\begin{align}\label{prop:modulus_compare:rho_zero}
\text{if $\gamma\notin \Gamma_0$, then}\,\,\, \sum_{\substack{i\in I\setminus J\\ |\gamma|\cap K_i\neq \emptyset}} \rho(K_i)=0.
\end{align}

For each $i\in J$ consider a point $x_i\in K_i$ and let $r_i=\diam(K_i)$, so that $K_i\subset \bar B(x_i,r_i)$. The $\tau$-fatness assumption implies that $\Sigma(B(x_i,2r_i))\leq c_1(\tau) \Sigma(K_i)$. We define a Borel function
$$\widetilde \rho =\rho\chi_{\Omega\setminus K}+ \sum_{i\in J} \frac{\rho(K_i)}{r_i}\chi_{B(x_i,2r_i)}.$$
Let $\gamma\in \Gamma\setminus \Gamma_0$ be locally rectifiable. By Lemma \ref{lemma:fat_count}, applied to the set $A=|\gamma|$, there exists a number $N(\tau)>0$ such that the set 
$$J_1=\{i\in J: K_i\cap |\gamma|\neq \emptyset\,\,\, \text{and}\,\,\, \diam(B(x_i,2r_i))\geq \diam(|\gamma|)\}$$ 
has at most $N(\tau)$ elements. Let $i\in J\setminus J_1$ be such that $K_i\cap |\gamma|\neq \emptyset$, so $|\gamma|$ intersects the ball $\bar B(x_i,r_i)$. Since $\diam(|\gamma|)>\diam(B(x_i,2r_i))$, the set $|\gamma|$ is not contained in $B(x_i,2r_i)$. This implies that
\begin{align*}
\int_\gamma \frac{\rho(K_i)}{r_i} \chi_{B(x_i,2r_i)}\, ds\geq \rho(K_i).
\end{align*}
We conclude that 
\begin{align}\label{prop:modulus_compare:rho_tilde}
\int_{\gamma}\widetilde \rho\,ds \geq \int_\gamma \rho\chi_{\Omega\setminus K}\, ds +\sum_{\substack{i\in J\setminus J_1\\ |\gamma|\cap K_i\neq \emptyset}} \rho(K_i).
\end{align}
Also, by \eqref{prop:modulus_compare:rho_m} and the bound on the cardinality of $J_1$, we obtain
$$\sum_{i\in J_1} \rho(K_i) \leq N(\tau) m^{1/2}.$$
We now prescribe $m=(2N(\tau))^{-2}$, so that the above sum is less than $1/2$. Given that $\rho$ is admissible for the transboundary modulus of $\Gamma$, by \eqref{prop:modulus_compare:rho_tilde} and \eqref{prop:modulus_compare:rho_zero} we have
$$\int_{\gamma}\widetilde \rho\, ds\geq 1/2.$$
Therefore $2\widetilde \rho$ is admissible for the conformal modulus of $\Gamma\setminus \Gamma_0$.  By the subadditivity of conformal modulus, the definition of $\widetilde \rho$, and Lemma \ref{lemma:bojarski}, we have
\begin{align*}
\Mod\Gamma&\leq \Mod\Gamma_0+\Mod (\Gamma\setminus \Gamma_0)=\Mod(\Gamma\setminus \Gamma_0) \leq 4\int \widetilde \rho^2\, d\Sigma\\
&\leq 8\int_{\Omega\setminus K} \rho^2\, d\Sigma+ 8\int  \left(\sum_{i\in J} \frac{\rho(K_i)}{r_i}\chi_{B(x_i,2r_i)} \right)^2\, d\Sigma\\
&\leq 8\int_{\Omega\setminus K} \rho^2\, d\Sigma+ c_2(\tau) \int  \left(\sum_{i\in J} \frac{\rho(K_i)}{r_i}\chi_{K_i} \right)^2\, d\Sigma\\
&= 8\int_{\Omega\setminus K} \rho^2\, d\Sigma+ c_2(\tau) \sum_{i\in J}\frac{\rho(K_i)^2}{r_i^2}\Sigma(K_i)\\
&\leq c_3(\tau) \left( \int_{\Omega\setminus K} \rho^2\, d\Sigma +\sum_{i\in I} \rho(K_i)^2 \right).
\end{align*}
Infimizing over $\rho$ gives the desired inequality with $c(\tau)\geq \max\{m^{-1},1,c_3(\tau)\}$.
\end{proof}

Given an open set $\Omega\subset \widehat \C$ and sets $E,F\subset \bar \Omega$, we denote by $\Gamma(E,F;\Omega)$ the family of curves $\gamma\colon[a,b]\to \bar \Omega$ such that $\gamma(a)\in E$, $\gamma(b)\in F$, and $\gamma((a,b))\subset \Omega$. Let $\Omega\subset \widehat \C$ be a domain. If there exists a decreasing function $\phi\colon (0,\infty)\to (0,\infty)$ such that
$$\Mod\Gamma(E,F;\Omega)\geq \phi(\Delta(E,F))$$
for every pair of disjoint non-degenerate continua $E,F\subset \bar \Omega$, then we call $\Omega$ a \textit{Loewner region}. The function $\phi$ is called the \textit{Loewner function} of $\Omega$ and we also say that $\Omega$ is a $\phi$-Loewner region.  In this work we will need the facts that the sphere $\widehat{\C}$ is a Loewner region (see \cite{HeinonenKoskela:qc}*{Section 6}) and that each quasidisk is a Loewner region, quantitatively. The latter follows from \cite{BonkHeinonenKoskela:gromov_hyperbolic}*{Remarks 6.6}, but it can be derived easily from Lemma \ref{lemma:quasimobius:cross} and the facts that the unit disk $\D$ is a Loewner region and that quasiconformal maps of the sphere are quasi-M\"obius (see \ref{q:qm_qc}).

We now obtain a modified version of \cite{Bonk:uniformization}*{Prop.\ 8.1}. The main difference is that Bonk's result assumes the continua of $\mathcal K$ to be uniformly quasiround and uniformly relatively separated, while here we use the condition of fatness. See also \cite{HakobyanLi:qs_embeddings}*{Theorem 4.3}.

\begin{proposition}[Loewner property of transboundary modulus]\label{prop:modulus_lower}
Let $\tau>0$. Let  $\Omega\subset \widehat \C$ be a Loewner region and $\mathcal K=\{K_i\}_{i\in I}$ be a collection of disjoint $\tau$-fat continua in $\Omega$  such that $K=\bigcup_{i\in I}K_i$ is relatively closed in $\Omega$. Then there exists a decreasing function $\phi\colon(0,\infty)\to (0,\infty)$ that depends only on the Loewner function of $\Omega$ and on $\tau$ with the following property. If $E,F\subset \bar\Omega$ are disjoint non-degenerate continua, then
$$\Mod_{\Omega,\mathcal K} \Gamma(E,F;\Omega) \geq \phi(\Delta(E,F)).$$
\end{proposition}

\begin{proof}
Suppose that $\Omega$ is a $\widetilde \phi$-Loewner region. In view of Proposition \ref{prop:modulus_compare}, we may take $\phi(t)= c(\tau)^{-1}\min\{1,\widetilde \phi(t)\}$.
\end{proof}

\subsection{Upper bound for transboundary modulus}\label{section:modulus_upper}

For an annulus $A$ in $\widehat \C$ we denote by $\Gamma(A)$ the family of curves $\gamma\colon[a,b]\to A$ such that $\gamma(a)$ and $\gamma(b)$ lie in distinct components of $\partial A$ and $\gamma((a,b))\subset \inter(A)$. The following result is a distilled version of \cite{Bonk:uniformization}*{Prop.\ 8.7}. 

\begin{proposition}[Transboundary modulus upper bound]\label{prop:modulus_upper}
Let $a,\tau>0$. Let $\mathcal K=\{K_i\}_{i\in I}$ be a countable collection of disjoint $\tau$-fat continua in $\widehat \C$ such that $K=\bigcup_{i\in I}K_i$ is a closed set. There exists a decreasing homeomorphism $\psi\colon (0,\infty)\to (0,\infty)$
that depends only on $a,\tau$ with the following property. If $A=A(x;r,R)$ is an annulus in $\widehat \C$ with $w_A(K_i)\leq a\cdot w_A^{1/3}$ for each $i\in I$, then  
$$\Mod_{\widehat \C, \mathcal K} \Gamma(A(x;r,R))\leq \psi(\log(R/r)).$$
\end{proposition}

Compare this inequality to the corresponding well-known estimate for conformal modulus (see \cite{Heinonen:metric}*{Lemma 7.18}):
$$\Mod\Gamma(A(x;r,R))\leq C (\log(R/r))^{-1}.$$  
Thus, Proposition \ref{prop:modulus_upper} asserts that if the uniformly fat continua $K_i$, $i\in I$, have small width relative to the annulus $A$, then the transboundary modulus of $\Gamma(A)$ enjoys an upper bound that is very similar to the upper bound of conformal modulus.

\begin{proof}
Let $A=A(x;r,R)$ be an annulus in $\widehat \C$ with $w_A(K_i)\leq a\cdot w_A^{1/3}$ for each $i\in I$. We define a transboundary mass distribution $\rho$ by
\begin{align*}
\rho(z)= \frac{\chi_A(z)}{w_{A} \sigma(x,z)}  \,\,\text{for}\,\, z\in \widehat \C\setminus K \quad \text{and}\quad \rho(K_i)=\frac{w_{A}(K_i)}{w_{A}} \,\, \text{for}\,\, i\in I.
\end{align*}
We show the admissibility of $\rho$ for $\Gamma(A)$. Let $\gamma\in \Gamma(A)$ be a curve that is locally rectifiable in $\widehat \C\setminus K$. Consider the function $\boldsymbol\pi\colon A \to [\log r,\log R]$ defined by $\boldsymbol\pi(z)=\log \sigma(x,z)$. Observe that
\begin{align}\label{prop:modulus_upper:pi}
[\log r,\log R] =\boldsymbol\pi(A)=\boldsymbol\pi(|\gamma|) = \boldsymbol\pi( |\gamma|\setminus K) \cup \bigcup_{i:K_i\cap |\gamma|\neq \emptyset} \boldsymbol\pi(A\cap K_i).
\end{align}
We break up the curve $\gamma$ into countably many non-overlapping pieces, each of which lies in $\widehat \C\setminus K$. For each such piece $\alpha$, an elementary inequality \cite{Bonk:uniformization}*{(37), p.~612} states that
$$ \int_{\alpha} \frac{1}{\sigma(x,\cdot )}\, ds \geq m_1(\boldsymbol\pi(|\alpha|)),$$
where $m_1$ denotes the $1$-dimensional Lebesgue measure. By applying this to each piece and then summing, we obtain
$$\int_{\gamma} \rho\chi_{\widehat \C\setminus K}\, ds= \frac{1}{w_{A}} \int_{\gamma} \frac{1}{\sigma(x,\cdot )} \chi_{\widehat \C \setminus K}\, ds\geq \frac{1}{w_{A}} m_1(\boldsymbol\pi( |\gamma|\setminus K)).$$
Also, for $i\in I$, by the definition of $w_{A}(K_i)$ we have
$$\rho(K_i)=\frac{w_{A}(K_i)}{w_{A}} \geq  \frac{1}{w_{A}} m_1(\boldsymbol\pi(A\cap K_i)).$$
Summarizing, 
\begin{align*}
\int_{\gamma} \rho\chi_{\widehat \C\setminus K}\, ds + \hspace{-1pt} \sum_{i: K_i\cap |\gamma|\neq \emptyset}\rho(K_i) &\geq \frac{1}{w_{A}}\bigg(m_1(\boldsymbol\pi( |\gamma|\setminus K))+ \hspace{-1pt} \sum_{i:K_i\cap |\gamma|\neq \emptyset}  m_1(\boldsymbol\pi(A\cap K_i))\bigg)  \\
&\geq \frac{1}{w_{A}} m_1(\boldsymbol\pi(|\gamma|))= 1,
\end{align*} 
where the last inequality follows from \eqref{prop:modulus_upper:pi} and the countable subadditivity of Lebesgue measure. The countability of the index set $I$ is used only here.

Finally, we produce a favorable upper bound for the mass of $\rho$.  Via integration in spherical coordinates, we have
\begin{align}\label{prop:modulus_upper:spherical}
\int_{A} \frac{d\Sigma}{\sigma(x,\cdot)^2}= 2\pi \int_{r}^{R} \frac{\sin \theta}{\theta^2}\, d\theta \leq 2\pi \int_{r}^{R} \frac{d\theta}{\theta}= 2\pi \cdot w_A.
\end{align}
Therefore,
\begin{align}\label{prop:modulus_upper:mass_cont}
\int_{\widehat \C\setminus K} \rho^2\, d\Sigma\leq \frac{1}{w_{A}^2} \int_{A} \frac{d\Sigma}{\sigma(x,\cdot)^2}\leq \frac{2\pi}{w_{A}}.
\end{align}
Regarding the discrete part of $\rho$, let $I_1$ be the set of indices $i\in I$ such that $A\cap K_i\neq \emptyset$ and $w_{A}(K_i)\leq \log 2$, and let $I_2$ be the set of $i\in I$ such that $A\cap K_i\neq \emptyset$ and $w_{A}(K_i)>\log 2$. Note that $\rho(K_i)=0$ if $A\cap K_i=\emptyset$, so 
$$\sum_{i\in I} \rho(K_i)^2=\sum_{i\in  I_1} \rho(K_i)^2+\sum_{i\in  I_2} \rho(K_i)^2.$$
If $w_{A}(K_i)>0$, we consider an annulus ${A(x;r_i,R_i)} \supset A\cap K_i$ such that $w_{A}(K_i)=\log(R_i/r_i)$. Suppose that $i\in I_1$ and $w_A(K_i)>0$, so $R_i\leq 2r_i$. By Lemma \ref{lemma:fat_annulus},
\begin{align*}
\bigg(\log\frac{R_i}{r_i}\bigg)^2\leq \left(\frac{R_i-r_i}{r_i}\right)^2\leq \frac{4(R_i-r_i)^2}{R_i^2}\leq c(\tau) \frac{\Sigma(A\cap K_i)}{R_i^2}\leq c(\tau)\int_{A\cap K_i} \frac{d\Sigma}{\sigma(x,\cdot)^2}.
\end{align*}
Therefore, by \eqref{prop:modulus_upper:spherical} we obtain
\begin{align}\label{prop:modulus_upper:mass_1}
\sum_{i\in I_1}\rho(K_i)^2 =\frac{1}{w_A^2}\sum_{i\in I_1}w_A(K_i)^2 \leq \frac{c(\tau)}{w_{A}^2}\int_{A}\frac{d\Sigma}{\sigma(x,\cdot)^2}\leq \frac{2\pi c(\tau)}{w_{A}}.
\end{align}
Next, we treat the sum over $I_2$. For $i\in I_2$ we have $2r_i<R_i$, so
\begin{align*}
\frac{1}{4}\leq \frac{(R_i-r_i)^2}{R_i^2}\leq c(\tau) \frac{\Sigma(A\cap K_i)}{R_i^2}\leq c(\tau)\int_{A\cap K_i} \frac{d\Sigma}{\sigma(x,\cdot)^2}.
\end{align*}
As a consequence, by \eqref{prop:modulus_upper:spherical} we have
$$\frac{1}{4}\# I_2 \leq c(\tau) \int_{A} \frac{d\Sigma}{\sigma(x,\cdot)^2}\leq  2\pi c(\tau) w_{A},$$
where $\# I_2$ denotes the cardinality of $I_2$, so
$$\#I_2 \leq 8\pi c(\tau)w_{A}.$$
By the assumption that $w_A(K_i)\leq a\cdot w_A^{1/3}$ for each $i\in I$, we conclude that
\begin{align}\label{prop:modulus_upper:mass_2}
\sum_{i\in I_2} \rho(K_i)^2 = \frac{1}{w_{A}^2} \sum_{i\in I_2}w_{A}(K_i)^2\leq \frac{1}{w_{A}^2} \#I_2 \cdot a^2w_{A}^{2/3}\leq \frac{8\pi a^2 c(\tau)}{w_{A}^{1/3}}. 
\end{align}
Combining \eqref{prop:modulus_upper:mass_cont}, \eqref{prop:modulus_upper:mass_1}, and \eqref{prop:modulus_upper:mass_2}, we see that the total mass of $\rho$ is at most
\begin{align*}
\frac{2\pi}{w_{A}}+ \frac{2\pi c(\tau)}{w_{A}}+ \frac{8\pi a^2 c(\tau)}{w_{A}^{1/3}} \leq c_1(a,\tau) \bigg(\bigg(\log \frac{R}{r}\bigg)^{-1}+ \bigg(\log\frac{R}{r}\bigg)^{-1/3}\bigg).
\end{align*}
This completes the proof with $\psi(t)= c_1(a,\tau)(t^{-1}+t^{-1/3})$, $t>0$.
\end{proof}

\section{Proof of the main theorem}

We first establish a preliminary version of Theorem \ref{theorem:main}. The proof of Theorem \ref{theorem:main} is given at the end of the section.

\begin{theorem}\label{theorem:main:finite}
Let $\{U_i\}_{i\in I}$ be a finite collection of Jordan regions in $\widehat \C$ that have disjoint closures and are  $K$-quasiconformally pairwise circularizable for some $K\geq 1$. Let $f\colon \widehat \C\to \widehat \C$ be a homeomorphism that is conformal in $G= \widehat \C\setminus \bigcup_{i\in I} \bar{U_i}$ and maps $\bar {U_i}$ to a closed disk $K_i$ for each $i\in I$. Then $f|_G$ extends to an $H(K)$-quasiconformal homeomorphism of $\widehat \C$.
\end{theorem}

The main ingredient for the proof of Theorem \ref{theorem:main:finite} is the next lemma, which gives an essential distortion estimate. Recall that a set $A$ in $\widehat \C$ separates two sets $E,F$ in $\widehat \C$ if $E$ and $F$ lie in distinct components of $\widehat \C\setminus A$.

\begin{lemma}\label{lemma:main:finite}
Let $f$ be as in Theorem \ref{theorem:main:finite}. Let $A$ be an annulus in $\widehat \C$ such that $f^{-1}(A)$ separates two disjoint non-degenerate continua $E,F\subset  G$. 
Then there exists an increasing function $\theta \colon (0,\infty)\to (0,\infty)$ that depends only on $K$ such that
$$w_A\leq \theta(\Delta(E,F)).$$
\end{lemma}

\begin{proof}
By Lemma \ref{lemma:subannulus}, there exists a subannulus $A'$ of $A$ that is concentric with $A$ such that $w_{A'}\geq \min\{w_A,w_A^{1/9}\}$ and one of the alternatives \ref{lemma:subannulus:1}, \ref{lemma:subannulus:2} holds. 

\textit{Case (i).} We first assume that alternative \ref{lemma:subannulus:1} of Lemma \ref{lemma:subannulus} holds. That is, $w_{A'}(K_i)\leq w_{A'}^{1/3}$ for all $i\in I\setminus I_0$, where $I_0=\emptyset$ or $I_0$ contains precisely one index of $I$. Let $\mathcal K'=\{K_i\}_{i\in I\setminus I_0}$, $\Omega'=\widehat \C\setminus \bigcup_{i\in I_0}K_i$,  $\mathcal U= \{\bar {U_i}\}_{i\in I\setminus I_0}$, and $\Omega=\widehat \C\setminus \bigcup_{i\in I_0}\bar{U_i}$. By the conformal invariance of transboundary modulus (Lemma \ref{lemma:invariance}) we have
\begin{align}\label{lemma:main:finite:qc}
\Mod_{\Omega, \mathcal U} \Gamma(E,F;\Omega) =\Mod_{\Omega', \mathcal K'} \Gamma(E',F';\Omega').
\end{align}
where $E'=f(E)$ and $F'=f(F)$. See the right part of Figure \ref{figure:modulus_large} for an illustration of $\Gamma(E',F';\Omega')$ when $I_0$ contains one index. Observe that
\begin{align}\label{lemma:main:finite:mod}
\Mod_{\Omega', \mathcal K'} \Gamma(E',F';\Omega') = \Mod_{\widehat \C, \mathcal K'} \Gamma(E',F';\Omega'),
\end{align}
since each curve in $\Gamma(E',F';\Omega')$ does not intersect the set $\widehat \C\setminus \Omega'$. By assumption, $f^{-1}(A)$ separates $E$ and $F$, so $A$ separates $E'$ and $F'$. Therefore, the subannulus $A'$ also separates $E'$ and $F'$. This implies that each curve in $\Gamma(E',F';\Omega')$ has a subcurve in $\Gamma(A')$, so
\begin{align}\label{lemma:main:finite:subordinate}
\Mod_{\widehat \C, \mathcal K'} \Gamma(E',F';\Omega') \leq \Mod_{\widehat \C, \mathcal K'}\Gamma(A').
\end{align}
By Proposition \ref{prop:modulus_upper}, based on the fatness of disks, there exists a universal decreasing homeomorphism $\psi_1\colon (0,\infty)\to (0,\infty)$ such that 
\begin{align}\label{lemma:main:finite:psi}
\Mod_{\widehat \C, \mathcal K'}\Gamma(A')\leq \psi_1(w_{A'})\leq \psi_1(\min\{w_A, w_A^{1/9}\}).
\end{align}
Combining \eqref{lemma:main:finite:qc}, \eqref{lemma:main:finite:mod}, \eqref{lemma:main:finite:subordinate}, and \eqref{lemma:main:finite:psi}, we obtain
\begin{align}\label{lemma:main:finite:upper}
\Mod_{\Omega, \mathcal U} \Gamma(E,F;\Omega)\leq \widetilde \psi_1(w_A),
\end{align}
where $\widetilde \psi_1(t)=\psi_1(\min\{t,t^{1/9}\})$, which is a decreasing homeomorphism of $(0,\infty)$.

The sets $U_i$, $i\in I$, are $K$-quasidisks by the assumption on quasiconformal circularization. By Lemma \ref{lemma:fat_quasidisk}, the sets $\bar{U_i}$, $i\in I$, are $\tau(K)$-fat. Moreover, either $\Omega=\widehat \C$, or $\Omega=\widehat \C\setminus \bar{U_{i_0}}$ for some $i_0\in I$; in the latter case $\Omega$ is a $K$-quasidisk. In both cases $\Omega$ is a Loewner region. We now apply Proposition \ref{prop:modulus_lower}, which yields a decreasing function $\phi_1\colon(0,\infty)\to (0,\infty)$ that depends only on $K$ with the property that
\begin{align*}
\Mod_{\Omega, \mathcal U} \Gamma(E,F;\Omega)\geq \phi_1(\Delta(E,F)).
\end{align*}
Combining this with \eqref{lemma:main:finite:upper}, we obtain $\phi_1(\Delta(E,F)) \leq \widetilde \psi_1(w_A)$. Since $\widetilde \psi_1$ is a decreasing homeomorphism of $(0,\infty)$, the function $\theta_1= (\widetilde \psi_1)^{-1}\circ \phi_1$ is increasing and we have
\begin{align}\label{lemma:main:finite:case1}
w_A \leq \theta_1 (\Delta(E,F)).
\end{align}

\textit{Case (ii).} Suppose that alternative \ref{lemma:subannulus:2} of Lemma \ref{lemma:subannulus} holds. That is, there exist $i_1,i_2\in I$ such that the disks $K_{i_1}$ and $K_{i_2}$ intersect both boundary components of the annulus $A'$. By assumption, there exists a $K$-quasiconformal homeomorphism $h\colon \widehat \C\to \widehat \C$ that maps $U_{i_1}$ and $U_{i_2}$ to disjoint disks $V_{i_1}$ and $V_{i_2}$, respectively. We set $V_i=h(U_i)$, $i\in I$. Consider the map $g= f\circ h^{-1}$, which maps $\bar{V_{i_1}}$ and $\bar{V_{i_2}}$ to $K_{i_1}$ and $K_{i_2}$, respectively, and is $K$-quasiconformal on $h(G)$. 

Let $\Gamma$ be the Schottky group of $\{\bar{V_{i_1}}, \bar{V_{i_2}}\}$, which is generated by reflections $\xi_1,\xi_2$ in $\bar{V_{i_1}}, \bar{V_{i_2}}$, respectively.  Let $\xi_j'$ be the reflection in $K_{i_j}$, $j=1,2$. The Schottky group $\Gamma'$ of $\{K_{i_1}, K_{i_2}\}$ is generated by $\xi_1',\xi_2'$. If $\xi=\xi_{i_1}\circ \dots\circ \xi_{i_n}\in \Gamma$ is expressed in reduced form, then $\xi'$ is defined to be $\xi_{i_1}'\circ \dots\circ \xi_{i_n}'$. Let $S=\widehat \C\setminus (V_{i_1}\cup V_{i_2})$ and $S_\infty= \bigcup_{\xi \in \Gamma}\xi(S)$. By Lemma \ref{lemma:schottky:limit} \ref{lemma:schottky:limit:1}, $\widehat \C\setminus S_\infty$ consists of two points $p_1,p_2$. By Lemma \ref{lemma:schottky:limit} \ref{lemma:schottky:limit:2} the map $g|_{S}$ extends to a homeomorphism $\widetilde g\colon \widehat \C\to \widehat \C$ such that $\xi' \circ \widetilde g = \widetilde g\circ \xi$ for each $\xi \in \Gamma$; see Figure \ref{figure:extension}. Moreover, parts \ref{lemma:schottky:limit:3} and \ref{lemma:schottky:limit:4} of Lemma \ref{lemma:schottky:limit} imply that the extension $\widetilde g$ is $K$-quasiconformal in the domain 
$$(\widehat\C\setminus \{p_1,p_2\}) \setminus \bigcup \{\xi(\bar{V_i}): i\in I\setminus \{i_1,i_2\},\,\, \xi\in \Gamma\}.$$ 

By Proposition \ref{prop:reflect} we have $w_{A'}(\xi'(K_i))\leq \beta_0$ for each $i\in I\setminus \{i_1,i_2\}$ and each $\xi'\in \Gamma'$. Note that if $w_{A'}^{1/3}>\beta_0$, then $w_{A'}(\xi'(K_i)) \leq w_{A'}^{1/3}$. On the other hand, if $w_{A'}^{1/3}\leq \beta_0$, then $w_{A'}(\xi'(K_i))\leq w_{A'} \leq \beta_0^{2}w_{A'}^{1/3}$. Hence, in both cases we have $w_{A'}(\xi'(K_i))\leq a\cdot w_{A'}^{1/3}$ for $a=\max\{1,\beta_0^2\}$, as required in Proposition \ref{prop:modulus_upper}. We define $\mathcal K'= \{ \xi'(K_i): i\in I\setminus \{i_1,i_2\},\,\, \xi'\in \Gamma'\}\cup \{p_1',p_2'\}$, where $p_j'=\widetilde g(p_j)$, $j=1,2$. By Proposition \ref{prop:modulus_upper}, there exists a decreasing homeomorphism $\psi_2\colon (0,\infty)\to (0,\infty)$ that depends only on $a$ (which is a universal constant) such that
\begin{align}\label{lemma:main:finite:psi2}
\Mod_{\widehat \C,\mathcal K'}\Gamma(A') \leq \psi_2(w_{A'})\leq \psi_2(\min\{w_A, w_A^{1/9}\}).
\end{align}

Each curve in $\Gamma(E',F';\widehat \C)$ has a subcurve in $\Gamma(A')$, so
\begin{align}\label{lemma:main:finite:subordinate2}
\Mod_{\widehat \C,\mathcal K'}\Gamma(E',F';\widehat \C)\leq  \Mod_{\widehat \C,\mathcal K'}\Gamma(A').
\end{align}
We let $\mathcal V=\{\xi(\bar {V_i}): i\in I\setminus \{i_1,i_2\},\,\, \xi\in \Gamma\}\cup \{p_1,p_2\}$, which satisfies $\widetilde g(\mathcal V)=\mathcal K'$ by the equivariance of $\widetilde g$ with respect to $\Gamma$. Since $\widetilde g$ is $K$-quasiconformal outside the union of the sets in $\mathcal V$, the quasi-invariance of transboundary modulus (Lemma \ref{lemma:invariance}) implies that
\begin{align}\label{lemma:main:finite:qc2}
\Mod_{\widehat \C,\mathcal V}\Gamma(h(E),h(F);\widehat \C)\leq K \cdot  \Mod_{\widehat \C,\mathcal K'}\Gamma(E',F';\widehat \C).
\end{align}
Note that $V_i$ and $\xi(V_i)$ are $K^2$-quasidisks for each $i\in I$ and $\xi\in \Gamma$, so the elements of $\mathcal V$ are $\tau(K)$-fat by Lemma \ref{lemma:fat_quasidisk}. By Proposition \ref{prop:modulus_lower} there exists a decreasing function $\phi_2\colon (0,\infty)\to (0,\infty)$ that depends only on $K$ such that
\begin{align}\label{lemma:main:finite:phi2}
\Mod_{\widehat \C,\mathcal V}\Gamma(h(E),h(F);\widehat \C)\geq \phi_2(\Delta(h(E),h(F))).
\end{align}
Therefore, by combining \eqref{lemma:main:finite:psi2}, \eqref{lemma:main:finite:subordinate2}, \eqref{lemma:main:finite:qc2}, and \eqref{lemma:main:finite:phi2}, we obtain
\begin{align*}
 \phi_2(\Delta(h(E),h(F)))\leq \widetilde \psi_2(w_A),
\end{align*}
where $\widetilde \psi_2(t)=K\cdot \psi_2(\min\{t,t^{1/9}\})$, which is a decreasing homeomorphism of $(0,\infty)$.

By property \ref{q:qm_qc}, the $K$-quasiconformal map $h$ is $\eta$-quasi-M\"obius, where $\eta$ depends only on $K$. By Lemma \ref{lemma:quasimobius:cross}, there exists a homeomorphism $\widetilde \eta\colon [0,\infty)\to [0,\infty)$ that depends only on $\eta$ such that 
$$\Delta(h(E),h(F))\leq \widetilde \eta(\Delta(E,F)).$$
Therefore, 
$$\phi_2(\widetilde \eta(\Delta(E,F))) \leq \phi_2(\Delta(h(E),h(F)))\leq \widetilde \psi_2(w_A).$$
The function $\theta_2= (\widetilde \psi_2)^{-1}\circ \phi_2\circ \widetilde \eta$ is increasing and we have
\begin{align}\label{lemma:main:finite:case2}
w_A \leq \theta_2 (\Delta(E,F)).
\end{align}

We finally note that by \eqref{lemma:main:finite:case1} and \eqref{lemma:main:finite:case2} the conclusion of the lemma holds in both Case (i) and Case (ii) for the increasing function $\theta=\max\{\theta_1,\theta_2\}$.
\end{proof}

We say that a metric space $(X,d)$ is \textit{linearly locally connected}\index{linearly locally connected space} or $\llc$\index{LLC space} for short if there exists a constant $M\geq 1$ such that for each ball $B_d(a,r)$ in $X$ the following two conditions hold.
\begin{enumerate}[label=\normalfont($\llc_{\arabic*}$)]
\item\label{llc1} For every $x,y\in B_d(a,r)$ there exists a continuum $E\subset B_d(a,Mr)$ that contains $x$ and $y$.
\item\label{llc2} For every $x,y\in X\setminus B_d(a,r)$ there exists a continuum $E\subset X\setminus B_d(a,r/M)$ that contains $x$ and $y$.
\end{enumerate}
In this case we say that $X$ is $M$-$\llc$. Every circle domain and every Schottky set is $1$-$\llc$. In the proof of Theorem \ref{theorem:main:finite} below we use this fact for finitely connected circle domains, which is established in \cite{Bonk:uniformization}*{Lemma 10.1}. See also \cite{BonkKleinerMerenkov:schottky}*{Prop.\ 2.2}, \cite{Rehmert:thesis}*{Prop.\ 4.1.3}, and the arXiv version of \cite{KarafylliaNtalampekos:gromov_hyperbolic}*{Lemma 2.13} for proofs in the general cases. 

A metric space $(X,d)$ is \textit{doubling} if there exists a constant $M\geq 1$ such that for every $d>0$, each set of diameter $d$ can be covered by at most $M$ sets of diameter at most $d/2$. Subsets of doubling spaces are also doubling with the same constant. We will use the fact that the Riemann sphere $\widehat \C$ and its subsets are doubling.

\begin{proof}[Proof of Theorem \ref{theorem:main:finite}]
By precomposing and postcomposing with M\"obius transformations, we may assume that $0,1,\infty\in G= \widehat \C\setminus \bigcup_{i\in I}\bar{U_i}$ and that $f$ fixes these three points. We will show that $f|_{G}$ is $\eta$-quasisymmetric and thus quasi-M\"obius (by \ref{q:qs_qm}) with a distortion function $\eta$ that depends only on $K$. Assuming this, Proposition \ref{prop:quasidisk_extension} applies to show that $f|_{G}$ extends to an $H(K)$-quasiconformal homeomorphism of $\widehat \C$, as desired.

We perform a further reduction. We use the notation $x'=f(x)$ for $x\in G$. Since $G$ is connected and $G,f(G)$ are doubling with a uniform parameter, being subsets of $\widehat \C$, by a theorem of V\"ais\"al\"a \cite{Heinonen:metric}*{Theorem 10.19}, it suffices to show that $f|_{G}$ is \textit{$H(K)$-weakly quasisymmetric}.  That is, given points $x,y,z\in  G$ with $\sigma(x,y)\leq \sigma(x,z)$, we have 
\begin{align}\label{theorem:main:finite:wqs}
\sigma(x',y')\leq H\sigma(x',z')
\end{align}
for a constant $H\geq 1$ depending only on $K$.

Let $x,y,z\in  G$ with 
\begin{align}\label{theorem:main:finite:assumption}
\sigma(x,y)\leq \sigma(x,z).
\end{align}
Without loss of generality, $x\neq z$, otherwise \eqref{theorem:main:finite:wqs} holds trivially. Let $a_1=0$, $a_2=1$, $a_3=\infty$. Note that $\sigma(a_i,a_j)\geq \pi/2$ for $i\neq j$. Thus, there exists at most one index $j\in \{1,2,3\}$ such that $\sigma(x',a_j')< \pi/4$. It follows that at least two indices $j\in \{1,2,3\}$ satisfy $\sigma(x',a_j')\geq \pi/4$. Among these two indices we choose one of them that satisfies $\sigma(y,a_j)\geq \pi/4$; again, the reverse inequality holds for at most one index. Without loss of generality, $j=1$, so we have
\begin{align}\label{theorem:main:finite:separation}
\sigma(x',a_1') \geq \pi/4 \quad \text{and} \quad \sigma(y,a_1)\geq \pi/4.
\end{align}

Observe that the image $f(G)$ is a circle domain, so it is $1$-$\llc$. Since $x',z'\in B(x', 2\sigma(x',z'))$, there exists a continuum $E'\subset B(x',2\sigma(x',z'))\cap f(G)$ joining $x'$ and $z'$. We define
$$\delta=\min \{\sigma(x',y'), \pi/4\}$$
and observe that
\begin{align}\label{theorem:main:finite:delta}
\sigma(x',y')\leq 4\delta.
\end{align}
We have $y',a_1'\notin B(x',\delta)$ in view of \eqref{theorem:main:finite:separation}. Thus, there exists a continuum $F'\subset  f(G)\setminus B(x', \delta)$ joining $y'$ and $a_1'$. Suppose that $3\sigma(x',z')\geq \delta/2$. Then $\sigma(x',y')\leq 4\delta \leq 24 \sigma(x',z')$, so \eqref{theorem:main:finite:wqs} holds with $H\geq 24$. Finally, we assume that $3\sigma(x',z')<\delta/2$. 

Consider the annulus $A=A(x'; 3\sigma(x',z'),\delta/2)$. The continua $E',F'$ lie in distinct components of $\widehat \C\setminus A$. Let $E=f^{-1}(E')$ and $F=f^{-1}(F')$. 
The continuum $E$ contains the points $x,z$ and the continuum $F$ contains the points $y,a_1$, so $\diam(F)\geq \pi/4$ by \eqref{theorem:main:finite:separation}. Observe that
\begin{align*}
\Delta(E,F)=\frac{\dist(E,F)}{\min\{\diam(E),\diam(F)\}}\leq \frac{\sigma(x,y)}{\min\{\sigma(x,z), \pi/4 \}} \leq 4,
\end{align*}
where the last inequality follows from \eqref{theorem:main:finite:assumption}.  By Lemma \ref{lemma:main:finite} there exists an increasing function $\theta\colon (0,\infty)\to (0,\infty)$ that depends only on $K$ such that 
\begin{align*}
w_A\leq \theta(\Delta(E,F)) \leq \theta(4).
\end{align*}
Therefore, by \eqref{theorem:main:finite:delta} we have
\begin{align*}
\sigma(x',y')\leq 4\delta = 8\cdot (\delta/2)=8e^{w_A}\cdot 3\sigma(x',z')\leq 24e^{\theta(4)}\sigma(x',z'). 
\end{align*}
This completes the proof with $H=24e^{\theta(4)}$.
\end{proof}

\begin{proof}[Proof of Theorem \ref{theorem:main}]
Suppose that $\{U_i\}_{i\in I}$ is a collection of at least two disjoint Jordan regions that are uniformly quasiconformally pairwise circularizable. Without loss of generality, $I\subset \N$. If $I$ is finite, consider the regions $U_i(n)$, $n\in \N$, $i\in I$, as provided by Proposition \ref{prop:disjoint_closures} and define $I_n=I$, $n\in \N$. If $I$ is countable, suppose that $I=\N$. Let $n\in \N$ be fixed. By Proposition \ref{prop:disjoint_closures}, for each $i\in \{1,\dots,n\}\eqqcolon I_n$, where $n\geq 2$, there exists a Jordan region $U_i(n)\subset U_i$ such that the regions $\{U_i(n)\}_{i\in I_n}$ are $H'(K)$-quasiconformally pairwise circularizable and have disjoint closures. Also, set $U_1(1)=U_1$. Moreover by Proposition \ref{prop:disjoint_closures} \ref{prop:disjoint_closures:1}, for each fixed $i\in I$ the regions $U_i(n)$ can be chosen so that for each compact set $E\subset U_i$ we have $E\subset U_i(n)$ for all sufficiently large $n\in \N$. This last statement also holds when $I$ is finite.

By Koebe's uniformization theorem \cite{Conway:complex2}*{Theorem 15.7.9}, for each $n\in \N$, there exists a conformal map $f_n$ from the finitely connected domain $G_n=\widehat \C\setminus \bigcup_{i\in I_n}\bar{U_i(n)}$ onto a circle domain in $\widehat \C$. The map $f_n$ extends uniquely to a homeomorphism of the closures; see \cite{Conway:complex2}*{Theorem 15.3.4}. In addition, with the aid of the Sch\"onflies theorem, we can extend $f_n$ non-uniquely to a homeomorphism of $\widehat{\C}$. By Theorem \ref{theorem:main:finite}, $f_n|_{G_n}$ extends to an $H(K)$-quasiconformal homeomorphism $F_n$ of $\widehat \C$. 

We now normalize $F_n$ as follows. We consider three points $a,b,c \in \widehat \C\setminus \bigcup_{i\in I} U_i \subset \widehat \C \setminus \bigcup_{i\in I_n} U_i(n) =\bar{G_n}$. By postcomposing $F_n$ with a M\"obius transformation, we may assume that $F_n$ fixes the points $a,b,c$. The sequence of $H(K)$-quasiconformal maps $\{F_n\}_{n\in \N}$ is normal \cite{LehtoVirtanen:quasiconformal}*{II.5} and, by passing to a subsequence, we assume that it converges uniformly to an $H(K)$-quasiconformal map $F\colon \widehat \C\to \widehat \C$. 

For each fixed $i\in I$, $F_n( U_i(n))$ is a disk for all $n\geq i$. Let $B_i$ be a compact Hausdorff limit of the closed disks $\bar{F_n(U_i(n))}$, $n\geq i$. Each $z\in U_i$ lies in $U_i(n)$ for all sufficiently large $n\in \N$, so we see that $F(U_i)\subset B_i$. Conversely, each point of $B_i$ is a limit of a sequence of points in $F_n(U_i(n))$, $n\in \N$, so $B_i\subset F(\bar{U_i})$. Therefore, $F|_{\bar{U_i}}$ is a homeomorphism from $\bar{U_i}$ onto $B_i$. A consequence of the invariance of domain theorem is that $F$ maps $U_i$ onto the interior $D_i$ of $B_i$. Since $B_i$ is a Hausdorff limit of closed disks, we conclude that $D_i$ is an open disk. Moreover, the disks $D_i$, $i\in I$, are disjoint, since $F$ is a homeomorphism. Thus, the set $T=\widehat \C\setminus \bigcup_{i\in I}D_i$ is a Schottky set and $F$ maps $S=\widehat \C\setminus \bigcup_{i\in I}U_i$ onto $T$.

Next, we show that $F$ is $1$-quasiconformal on $S$. For each $n\in \N$ the map $F_n$ is conformal on $G_n$. Each boundary component of $G_n$ is a quasicircle, so it has measure zero. Hence, $F_n$ is $1$-quasiconformal on $\bar {G_n}$. Since $\bar G_n\supset S$, $F_n$ is $1$-quasiconformal on $S$. By Lemma \ref{lemma:lsc}, the limit $F$ is $1$-quasiconformal on $S$.

We finally prove the ultimate uniqueness statement in Theorem \ref{theorem:main}. Let $\widetilde F$ be another quasiconformal map of $\widehat \C$ that maps $S$ onto a Schottky set $\widetilde T$ and is $1$-quasiconformal on $S$. Let $g=(\widetilde F\circ F^{-1})|_T$. We need to show that $g$ is the restriction of a M\"obius transformation of $\widehat \C$. When $T$ has area zero, this follows from a theorem of Bonk--Kleiner--Merenkov \cite{BonkKleinerMerenkov:schottky}*{Theorem 1.1}. 

Suppose that $T$ has positive area. Let $\xi_i$ be the reflection in $\bar{D_i}$ and $\xi_i'$ be the reflection in $g(\bar{D_i})$, $i\in I$. Let $\Gamma_{T}$ be the Schottky group of $\{\bar{D_i}\}_{i\in I}$, which is generated by $\{\xi_i\}_{i\in I}$. For $\xi=\xi_{i_1}\circ \dots\circ \xi_{i_n}$, expressed in reduced form, we define $\xi'=\xi_{i_1}'\circ \dots\circ \xi_{i_n}' \in \Gamma_{\widetilde T}$. As proved in \cite{BonkKleinerMerenkov:schottky}*{Prop.\ 5.5}, $g$ has an extension to a quasiconformal homeomorphism of $\widehat \C$, which we also denote by $g$, that is equivariant with respect to $\Gamma_T$; that is $g\circ \xi = \xi'\circ g$ for each $\xi\in \Gamma_T$.

The proof of the main result of \cite{BonkKleinerMerenkov:schottky} (see {p.~426}), even in the case that $T$ has positive area, implies that the extension $g$ is $1$-quasiconformal on $\widehat \C\setminus T_\infty$, where $T_\infty=\bigcup_{\xi\in \Gamma_T} \xi(T)$. We sketch the argument for the convenience of the reader. Since $g$ is quasiconformal in $\widehat \C$, it is differentiable with invertible derivative at almost every point in $\widehat \C\setminus T_\infty$. Each point of $\widehat \C\setminus T_\infty$ lies in a sequence of nested closed balls $\{A_j\}_{j\in \N}$ with diameters shrinking to $0$ such that each ball is the image of a ball in the collection $\{\bar{D_i}\}_{i\in I}$ under an element of $\Gamma_T$; see \cite{BonkKleinerMerenkov:schottky}*{Lemma 3.4}. The equivariance of $g$ with respect to $\Gamma_T$ implies that $g(A_j)$ is also a ball for each $j\in \N$. Thus, at points of differentiability in $\widehat \C\setminus T_\infty$, $g$ maps infinitesimal balls to infinitesimal balls, which implies that it is $1$-quasiconformal; see \cite{BonkKleinerMerenkov:schottky}*{Lemma 6.1} for a precise statement.

By Lemma \ref{lemma:conformal_composition}, $F^{-1}$ is $1$-quasi\-conformal on $T$ and $g|_T=\widetilde F\circ F^{-1}$ is $1$-quasi\-conformal on $T$. Also, by Lemma \ref{lemma:conformal_composition} and Remark \ref{remark:conformal_composition}, $\xi'\circ g\circ \xi^{-1}$ is $1$-quasiconformal on $\xi(T)$ for each $\xi\in \Gamma_T$. Thus, $g$ is $1$-quasiconformal on $\xi(T)$ for each $\xi\in \Gamma_T$. We conclude that $g$ is $1$-quasiconformal on $T_\infty$. Altogether, $g$ is $1$-quasiconformal on $\widehat \C$. By Weyl's lemma \cite{LehtoVirtanen:quasiconformal}*{Theorem I.5.1}, $g$ is conformal on $\widehat \C$ so it is a M\"obius transformation. 
\end{proof}

\bibliography{../../../biblio} 

\end{document}